\documentclass[preprint,12pt]{elsarticle}

\usepackage{amssymb}

\usepackage{graphicx}%
\usepackage{multirow}%
\usepackage{amsmath,amssymb,amsfonts}%
\usepackage{amsthm}%
\usepackage{mathrsfs}%
\usepackage[title]{appendix}%
\usepackage{xcolor}%
\usepackage{textcomp}%
\usepackage{manyfoot}%
\usepackage{booktabs}%
\usepackage{algorithm, algpseudocode}
\usepackage{algorithmicx}%
\usepackage{algpseudocode}%
\usepackage{mathtools}
\makeatletter
\def\plist@algorithm{Alg.\space}
\makeatother

\newcommand{\Matlab}{\textsc{Matlab}}

\newcommand{\mlin}[1]{\mbox{\tt{#1}}}

\usepackage{enumitem}

\usepackage{nicematrix}

\usepackage{float}

\usepackage[breaklinks]{hyperref}
\usepackage{breakurl}
\usepackage{listings}%
\usepackage{cleveref}

\usepackage{tikz}
\usetikzlibrary{patterns}
\usetikzlibrary{shapes,arrows}
\usetikzlibrary{intersections}
\usetikzlibrary{decorations.pathreplacing}
\usepackage{tcolorbox}
\usetikzlibrary{calc}
\definecolor{mpiblue}{HTML}{33A5C3}
\colorlet{MPIblue}{mpiblue}
\usepackage{tkz-euclide}
\usetikzlibrary{positioning,fit}
\usetikzlibrary{shapes.geometric}

\usepackage{pgfplots}
\usepgfplotslibrary{patchplots}
\tikzset{>=latex}
\usepackage{comment}
\pgfplotsset{width=7cm,compat=1.8}
\definecolor{mpibluefont}{HTML}{17A1C1}
\colorlet{MPIbluefont}{mpibluefont}
\definecolor{mpigreen}{HTML}{007675}
\colorlet{MPIgreen}{mpigreen}
\definecolor{mpired}{HTML}{78004B}
\colorlet{MPIred}{mpired}
\definecolor{mpisand}{HTML}{ece9d4}
\colorlet{MPIsand}{mpisand}

\usetikzlibrary{patterns,angles,quotes}

\crefname{table}{tab.}{tabs.}
\Crefname{table}{Table}{Tables}

\usepackage{extramath}
\usepackage{setspace}
\usepackage{subcaption}
\usepackage{float}
\usepackage{amssymb}

\newtheorem{theorem}{Theorem}[section]
\newtheorem{proposition}[theorem]{Proposition}%
\newtheorem{corollary}[theorem]{Corollary}%

\journal{Applied Mathematics and Computation}

\begin{document}

\begin{frontmatter}



\title{Generalized cyclic symmetric decompositions for the matrix multiplication tensor}

\author[label1]{Charlotte Vermeylen}
\affiliation[label1]{organization={Department of Electrical Engineering (ESAT) KU Leuven},
             addressline={Kasteelpark Arenberg 10},
             city={Heverlee},
             postcode={3001},
             state={Vlaams-Brabant},
             country={Belgium}}
             
\author[label2]{Marc Van Barel}
\affiliation[label2]{organization={Department of Computer Science KU Leuven},
             addressline={Celestijnenlaan 200C A},
             city={Heverlee},
             postcode={3001},
             state={Vlaams-Brabant},
             country={Belgium}}


\begin{abstract}
A new generalized cyclic symmetric structure in the factor matrices of polyadic decompositions of matrix multiplication tensors for non-square matrix multiplication is proposed to reduce the number of variables in the optimization problem and in this way improve the convergence. The structure is implemented in an existing numerical optimization algorithm. Extensive numerical experiments are given that the proposed structure indeed finds more (practical) decompositions.
\end{abstract}

%


\end{frontmatter}


\section{Introduction} \label{sec:introduction}
Matrix multiplication is a fundamental operation in (numerical) linear algebra and appears in numerous applications ranging from machine learning, signal processing, robotics, etc. The discovery of faster algorithms can therefore have a significant impact. \emph{Fast matrix multiplication (FMM)} has already been investigated for many years since the first FMM algorithm was discovered by Strassen in 1969. However, still many open problems remain. Recently, DeepMind has discovered faster FMM algorithm using deep reinforcement learning \cite{DeepMind_2022} and more recently, an \emph{Augmented Lagrangian (AL)} algorithm was proposed to find more stable and faster FMM algorithms \cite{vermeylen2023stability}. In this work, we build upon the algorithm from \cite{vermeylen2023stability} by incorporating a new structure and show in various numerical experiments that including this structure facilitates the discovery of new \emph{practical} FMM algorithms.

We consider the \emph{fast matrix multiplication (FMM)} problem, which rewrites the bilinear equations of matrix multiplication,
\begin{equation} \label{eq:matrix_mult2}
C(i,j) = \sum_{k=1}^p A(i,k) B(k,n),
\end{equation}
where $C:=AB$, as a tensor equation:
\begin{align} \label{eq:matrix_mult_tensor1}
C &= \sum_{i,j,k=1}^{m,p,n} \big\langle E_{ik}^{m\times p}, A  \big\rangle_\mathrm{F} \big\langle E_{kj}^{p\times n}, B \big\rangle_\mathrm{F} E_{ij}^{m\times n} \nonumber \\
&= \left( \sum_{i,j,k=1}^{m,p,n} E_{ik}^{m\times p} \op E_{kj}^{p\times n} \op E_{ij}^{m\times n} \right)\cdot_\mathrm{F} \left( A, B \right),
\end{align}
where $E_{ij}^{m\times p}$, for $i=1,\dots,m$ and $j=1,\dots, p$, is a basis of matrices in $\mathbb{R}^{m\times p}$ such that $E_{ij}^{m\times p}(i_1,i_2) = 1$ if $(i_1,i_2)=(i,j)$ and zero else, `$\op$' denotes the tensor product and $\langle \cdot \rangle_\mathrm{F}$ and `$\cdot_\mathrm{F}$' denote the Frobenius inner product. A transpose is added to \eqref{eq:matrix_mult_tensor1} such that the tensor that is defined has additional interesting properties, such as \emph{cyclic symmetry (CS)}, which will be discussed in more detail in \Cref{sec:CS}: 
\begin{align} \label{eq:matrix_mult_tensor}
C^\top &= \left( \sum_{i,j,k=1}^{m,p,n} E_{ik}^{m\times p} \op E_{kj}^{p\times n} \op \left(E_{ij}^{m\times n}\right)^\top \right)\cdot_\mathrm{F} \left( A, B \right) \nonumber \\
&= \left( \sum_{i,j,k=1}^{m,p,n} E_{ik}^{m\times p} \op E_{kj}^{p\times n} \op E_{ji}^{n\times m} \right)\cdot_\mathrm{F} \left( A, B \right).
\end{align}
Note that the Frobenius inner product is a linear function in the elements of a matrix. The rank of the bilinear equation is the minimal $r$  such that the equation can be written as
\begin{align} \label{eq:base_alg}
C^\top &=  \sum_{i=1}^{r} \big\langle U_i, A  \big\rangle_\mathrm{F} \big\langle V_i, B \big\rangle_\mathrm{F} W_i,
\end{align}
for some matrices $U_i \in \mathbb{R}^{m \times p}$, $V_i \in \mathbb{R}^{p \times n}$, and $W_i \in \mathbb{R}^{n \times m}$. It can be shown that minimizing $r$ minimizes the computational complexity of matrix multiplication \cite[Proposition 15.1]{burgisser2013algebraic}, \cite{landsberg2011tensors}. More specifically, the number of arithmetic operations needed to compute two matrices in $\mathbb{R}^{n \times n}$ is $\mathcal{O}\left( n^\omega \right)$, where $\omega:= \log_n r$. Since $ r < n^3$, it holds that $\omega < 3$. Additionally, it is well known that the exponent is bounded from below by two, since you need at least $n^2$ operations to compute all $n^2$ elements of the matrix multiplication.

Minimizing the rank of the bilinear equation corresponds to finding a \emph{canonical polyadic decomposition (CPD)} of the \textit{matrix multiplication tensor (MMT)} \cite{Strassen1969} defined implicitly in \eqref{eq:matrix_mult_tensor}:
\begin{align} \label{eq:MMT_triv}
T_{mpn} &:=  \sum_{i,j,k=1}^{m,p,n} e_{ik}^{mp} \op e_{kj}^{pn} \op e_{ji}^{mn},
\end{align}
where $e_{ik}^{mp}:= \mathrm{vec} \left(E_{ik}^{m\times p} \right)$ and $T_{mpn}$ is the MMT of size $mp \times pn \times mn$. 
Because of its definition, the MMT is a tensor consisting of only $mpn$ ones: 
\begin{equation} \label{eq:ones_MMT}
\begin{split}
&T_{mpn}(i_1 + (i_2-1)m,j_1 + (j_2-1)p,k_1 + (k_2-1)n) \\
&= \begin{cases} 1 & \mathrm{ if }~ i_1 = k_2, i_2 = j_1,j_2 = k_1 \\
0 & \mathrm{else}
\end{cases},
\end{split}
\end{equation}
for all $i_1,k_2 = 1, \dots, m$, $i_2,j_1 = 1, \dots, p$, and $j_2, k_1 = 1, \dots, n$.
Note that an MMT only depends of the size of the matrices that are multiplied and not on the values. 

A PD of $T_{mpn}$ of \emph{length} $r$ decomposes $T_{mpn}$ into $r$ rank-1 tensors:
\begin{equation} \label{eq:CPD}
\begin{split}
T_{mpn} = \sum_{i=1}^r u_i \op v_i \op w_i,
\end{split}
\end{equation}
where $u_i,v_i$, and $w_i$ are vectors in $\mathbb{R}^{n_1}$, $\mathbb{R}^{n_2}$, and $\mathbb{R}^{n_3}$ respectively. These vectors can be stored in three so-called \emph{factor matrices} $U$, $V$, and $W$:
$U := \left[ u_{1}, \cdots, u_r \right]$, $V := \left[ v_{1}, \cdots, v_r \right]$, and $W := \left[ w_{1}, \cdots, w_{r} \right]$. 
A PD of length or \emph{rank} $r$ is denoted by $\mathrm{PD}_r$ and a $\mathrm{PD}_r$ of a tensor $T$ as $\mathrm{PD}_r(T)$. The minimal rank for which a PD of a tensor $T$ exists is denoted by the rank $r^*$ of this tensor or $r^*(T)$. The PD is then called canonical (CPD) and denoted by $\mathrm{PD}_{r^*}(T)$. 

A $\mathrm{PD}_r$ \eqref{eq:CPD} of $T_{mpn}$ can be used to obtain a \emph{base algorithm} for the multiplication of two matrices $A \in \mathbb{R}^{m\times p}$ and $B \in \mathbb{R}^{p \times n}$:
\begin{align} \label{eq:vecAB_TM}
\text{vec}\left( C^\top \right) &= \sum_{i=1}^{r} \langle u_i, \text{vec}(A) \rangle \langle v_i,  \text{vec}(B) \rangle w_i.
\end{align}
Note that this is the same equation as \eqref{eq:base_alg} with $U_i = \text{reshape}(u_i,[m,p])$, $V_i = \text{reshape}(v_i,[p,n])$, and $W_i = \text{reshape}(w_i,[n,m])$, for all $i=1, \dots, r$. The number of \emph{active} multiplications, i.e., multiplication between linear combinations of elements of $A$ and $B$ is reduced to $r$. The upper bound $mpn$ holds for standard matrix multiplication. When a base algorithm is applied recursively to multiply large matrices, the active multiplication determine the asymptotic complexity \cite{Landsberg_complexity}.

One of the difficulties is that the rank of a CPD of $T_{mpn}$, denoted by $r^*(T_{mpn})$, is not known for most $m$, $p$, and $n$. One of the counterexamples is $T_{222}$, for which the rank $r^*$ is known to be seven \cite{DeGroote1978b} and the problem is fully understood.  
We call $\mathrm{PD}_{7}(T_{222})$ discovered by Strassen $\mathrm{PD}_{\mathrm{Strassen}}$ \cite{Strassen1969}. For other $m$, $p$, and $n$, only lower and upper bounds on the rank exist. For example, it is known that $r^*(T_{333})$ is higher than or equal to 19 \cite{Blaser2003}, and smaller than or equal to 23 \cite{Laderman1976, Ballard2019, Smirnov2013, new_ways_33_2021}. We denote the upper bound with $\tilde{r}(T_{mpn})$, i.e., the lowest rank for which a PD of $T_{mpn}$ is known. An overview of $\tilde{r}(T_{mpn})$ for different $m$, $p$, and $n$, is for example given in \cite[Table 1]{Smirnov2013} and \cite[Fig. 1]{DeepMind_2022}. Note that the rank of a MMT is invariant under permutation of $m$, $p$, and $n$.

To find a $\mathrm{PD}_r$ of $T_{mpn}$, we use the following \emph{nonlinear least squares (NLS)} cost function, which is also frequently used in the literature for the FMM problem \cite{Smirnov2013,Tichavsky2017,Ballard2019,Chiantini2019}:
\begin{equation} \label{eq: unconstr_LSQ_min}
\min_x  \underbrace{\frac{1}{2} \left\lVert  F(x)- \text{vec} \left(T_{mpn} \right) \right\rVert^2}_{=: f(x)},
\end{equation}
where $\lVert \cdot \rVert $ denotes the {$\ell_2$}-norm, and $F(x)$ is defined as
\begin{align} \label{eq:F(x)}
F\left(x \right) &:=  \mathrm{vec} \left( \sum_{r=1}^r u_r \op v_r \op w_r \right), & x &:= \begin{bmatrix}
\mathrm{vec}(U)\\ \mathrm{vec}(V)\\ \mathrm{vec}(W)  
\end{bmatrix}.
\end{align}
Note that $f(x^*)=0$, if and only if $x^*$ is a $\mathrm{PD}_r$ of $T_{mpn}$. However, most standard optimization algorithms fail to converge to a global optimum of \eqref{eq: unconstr_LSQ_min}. One of the reasons is that PDs of MMT have more invariances than generic PDs, which we will briefly review in \Cref{sec:inv_transf} \cite{DeGroote1978a}. This means that the minima of \eqref{eq: unconstr_LSQ_min} are non-isolated.

In practice, the minimal length or rank ${r}^*$ is approximated experimentally by the lowest length $\tilde{r}$ for which a solution $x^*$ with $f(x^*)=0$ can be obtained using numerical optimization. However, no globally convergent algorithm for \eqref{eq: unconstr_LSQ_min} is known to us, and thus we can never be sure that no solution with a lower rank exists. 

Another difficulty when solving \eqref{eq: unconstr_LSQ_min} is that number of parameters or unknowns grows rapidly with $m$, $p$, and $n$, Consequently, solving \eqref{eq: unconstr_LSQ_min} is only feasible for small $m$, $p$, and $n$, e.g., $m,p,n \leq 4$.
Therefore, in this paper a new structure is proposed that can be enforced in PDs of MMTs to reduce the number of variables and hereby speed-up the convergence.

To solve \eqref{eq: unconstr_LSQ_min}, in the literature the alternating least squares (ALS) method is frequently used \cite{Smirnov2013,params_R23_1986,Chiantini2019,Ballard2019} but the convergence is usually slow and a lot of starting points are needed to obtain a solution in a reasonable amount of time. In \cite{Tichavsky2017} and \cite{Smirnov2013}, a constrained optimization problem and corresponding method are proposed to improve the convergence. More specifically, a Levenberg-Marquardt (LM) method with Lagrange parameters and a quadratic penalty (QP) term in combination with the ALS method are used respectively. However, both methods were still not able to find PDs of rank 49 of MMTs for $4 \times 4$ matrix multiplication whereas these PDs are known to exist. That is why we proposed in \cite{vermeylen2023stability} an augmented Lagrangian (AL) method and a new constrained optimization problem to find PDs of MMTs:
\begin{equation}\label{eq:constr_min}
\min_x f(x), \quad \mathrm{s.t.} \quad h(x)=0, \quad l \leq x \leq u.
\end{equation}
Different equality constraints $h(x)$ were proposed.
The AL method can be considered as a combination of the LM and QP method since the AL objective function is:
\begin{equation*} 
\begin{split}
\min_{x,y_1,y_2,z}  \underbrace{f(x) + \left\langle y_1,h(x) \right\rangle + \langle y_2, x - z \rangle + \frac{\beta}{2} \left( \lVert h(x) \rVert^2 + \lVert x- z\rVert^2 \right)}_{=:\mathcal{L}_A \left( x,y_1, y_2,z;\beta \right)},
\end{split}
\end{equation*}
subject to (s.t.) $l \leq z \leq u$, where $y_1$ and $y_2$ are vectors of Lagrange multipliers, $\beta$ is a regularization parameter, and $z$ is a vector of slack variables. Because $f(x)$ is an NLS function, the AL objective can also be rewritten as an NLS function. The variables in $x$ and the auxiliary parameters are updated successively according to \cite[Section 17.4]{Nocedal12006}. We use the LM method for the minimization to $x$. The pseudo-code of the algorithm is shown in \cite[Algorithm 4.2]{vermeylen2023stability} and the {\Matlab} implementation is publicly available \footnote{\url{https://github.com/CharlotteVermeylen/ALM_FMM_stability}}.
The advantage of the AL method is that the constraints are satisfied accurately, even when we are not in the neighborhood of an optimum. This method was used to obtain new $\mathrm{PD}_{49} \left(T_{444} \right)$s and different new CS PDs. This method is used as well in the numerical experiments in \Cref{sec:num_exp}.

The rest of the paper is organized as follows. \Cref{sec:prelim} gives some preliminaries concerning the FMM problem. In \Cref{sec:gen_CS}, the new generalized CS structure is proposed, and in \Cref{sec:num_exp}, numerical experiments are given to demonstrate the use and advantages of including this structure in the optimization problem \eqref{eq: unconstr_LSQ_min}.

\section{Preliminaries} \label{sec:prelim}

This section first discusses what is meant with practical PDs and FMM algorithms. Afterwards, we give some well known properties of MMTs.

\subsection{Practical algorithms}
As discussed in the introduction, a base algorithm corresponding to a $\mathrm{PD}_r(T_{mpn})$ can be applied recursively to multiply larger matrices, e.g., of size $m^k \times p^k$ and $p^k \times n^k$. Such an algorithm requires $\mathcal{O} \left( c r^k\right)$ floating point operations, where the constant $c$ depends on the number of elements in the factor matrices. When $c$ is sufficiently small and $k$ sufficiently large, fewer operations are required compared to standard matrix multiplication, i.e., $\mathcal{O}\left((mpn)^k\right)$. To decrease the constant $c$, sparse PDs are desired with elements that are powers of 2, because multiplication with a power of 2 is not a costly operation when implemented in hardware. These PDs are called \emph{practical PDs} because they speed up standard matrix multiplication for large matrices and consequently are useful in various practical applications. 

Consequently, if a numerical PD of $T_{mpn}$ is obtained, it has to be transformed to obtain a practical algorithm. The \emph{invariance transformations} or \emph{inv-transformations} \cite{DeGroote1978a}, which are summarized in \Cref{sec:inv_transf}, can be used for this purpose \cite{Tichavsky2017}. However, not all numerical PDs of $T_{mpn}$ can be transformed in this way \cite{berger2022equivalent}. In \cite{berger2022equivalent}, this process is called \emph{discretization}.

A disadvantage of \eqref{eq: unconstr_LSQ_min} is that the solutions have floating point elements and thus extra steps have to be taken or constraints have to be added to the optimization problem to obtain practical PDs. That is why we proposed a new constrained problem formulation in \cite{vermeylen2023stability} with equality constraint:
\begin{equation} \label{eq:h_discr}
h_{\mathrm{discr}} (x):= x\cdot (x-1) \cdot (x+1),
\end{equation}
where `$\cdot$' denotes element-wise multiplication.

\subsection{Properties of the matrix multiplication tensor}

In this section, first the invariances that are present for PDs of $T_{mpn}$ are discussed. A similar overview already appeared in, e.g., \cite{vermeylen2023stability}.
Afterwards, cyclic symmetry and recursive PDs are discussed in more detail. 

\subsubsection{Invariance transformations} \label{sec:inv_transf}

The parametrization of a PD using factor matrices is well known to be not unique. For example, the same tensor is obtained when permuting the rank-1 tensors, and consequently the factor vectors.
Additionally, because of the multi-linearity of the tensor product, the columns can be scaled as: $\alpha_i u_i, \beta_i v_i, \frac{1}{\alpha_i \beta_i} w_i$, for all $i=1, \dots, r$, and $\alpha_i,\beta_i$ in $\mathbb{R}_0$. 
Note that only the scaling invariance is a continuous invariance transformation (of dimension $2r$).

PDs of $T_{mpn}$ also have additional invariances \cite{DeGroote1978a}. For example, it is well known that they are invariant under the following \emph{PQR-transformation}:
\begin{equation} \label{eq:mult_inv}
\begin{split}
u_i' \leftarrow \text{vec} \left( P U_i Q^{-1} \right), \quad v_i' \leftarrow \text{vec} \left( Q V_i R^{-1} \right),  \quad w_i' \leftarrow \text{vec} \left( R W_i P^{-1} \right), \\
\end{split}
\end{equation}
for $i=1, \dots, r$, and where $P \in GL(m)$, $Q \in GL(p)$, and $R \in GL(n)$, where $GL(i)$ represents the \emph{general linear group} of invertible matrices in $\mathbb{R}^{i \times i}$. 
This invariance can easily be proven using the equality: $C := AB = \left(P^\top \right)^{-1} \left( P^\top A \left(Q^\top \right)^{-1} \right) \left( Q^\top B \left(R^\top \right)^{-1} \right) R^\top$. When we substitute this in \eqref{eq:vecAB_TM}, we obtain:
\begin{equation*}
\begin{split}
C = \sum_{i=1}^r &\text{trace} \left( U_i \left(P^\top A \left(Q^\top \right)^{-1} \right)^\top \right) \\
&\text{trace} \left( V_i \left( Q^\top B \left(R^\top \right)^{-1} \right)^\top \right) \left(P^\top \right)^{-1} W_i^\top R^\top\\
= \sum_{i=1}^r &\text{trace} \left( P U_i Q^{-1} A^\top \right) \text{trace} \left( Q V_i R^{-1} B^\top \right) \left( R W_i P^{-1} \right)^\top,
\end{split}
\end{equation*}
where we made use of the fact that the trace is invariant under cyclic permutation of the factors. And thus indeed another base algorithm or $\mathrm{PD}_r$ of $T_{mpn}$ can be obtained using \eqref{eq:mult_inv}. 
Additionally, when $m=n=p$, PDs of $T_{mmm}$ are invariant under the following \emph{transpose-transformation}:
\begin{equation*}
\begin{split}
u_i' \leftarrow \text{vec} \left( V_i^\top \right) , \quad v_i' \leftarrow \text{vec} \left( U_i^\top \right) , \quad w_i' \leftarrow \text{vec} \left( W_i^\top \right),\\
\end{split}
\end{equation*}
for $i=1, \dots, r$, which can be proven using the fact that $C = \left( B^\top A^\top \right)^\top$. 
Lastly, when $m=n=p$, the PDs are invariant under cyclic permutation of the factor matrices.

The PQR-invariance is a continuous invariance and it overlaps with the scaling invariance when $P=aI$, $Q=bI$, and $R=cI$, where $I$ is the identity matrix of the appropriate size and $a,b$, and $c$ are scaling factors in $\mathbb{R}_0$. Thus, the combination of the scaling and PQR-transformation has at most dimension $2r + m^2 + p^2 + n^2 - 3$ for all PDs of $T_{mpn}$.


The combination of the invariance transformations described above is called an \emph{inv-transformation}. Two PDs of $T_{mpn}$ that can be obtained using inv-transformations are called \emph{inv-equivalent}. Only for $T_{222}$, it is known that all PDs are inv-equivalent \cite{DeGroote1978b}.

\subsection{Jacobian matrix}

In \cite{vermeylen2023stability}, it is shown that the Jacobian matrix can be used to investigate the inv-equivalence of PDs of MMTs and to investigate any additional invariances, such as the ones discovered in \cite{vermeylen2023stability} using parametrizations of decompositions. That is why in \Cref{sec:num_exp}, the size and ranks of the Jacobian matrix at solutions with the proposed structure, are given. 

\subsubsection{Cyclic symmetry} \label{sec:CS}

Because of \eqref{eq:MMT_triv}, $T_{mpn}$ is a structured tensor. More specifically, when $m=p=n$, $T_{mpn}$ is a so-called cyclic symmetric (CS) tensor, which means that $T_{mmm}(i,j,k) = T_{mmm}(j,k,i)=T_{mmm}(k,i,j)$, for all $i,j,k=1,\dots, m^2$. 
This structure can be used in the parameterization of a PD. More specifically,  it can be enforced that the rank-1 tensors are symmetric or occur in CS pairs \cite{ballard2017_CS}:
\begin{equation*}
T_{mmm} = \sum_{i=1}^s a_i \op a_i \op a_i +\sum_{j=1}^t \left( b_j \op d_j \op c_j + c_j \op b_j \op d_j + d_j \op c_j \op b_j \right),
\end{equation*}
where $a_i,b_j,c_j$, and $d_j$ are vectors in $\mathbb{R}^{m^2}$, for all $i =1,\dots, s$ and $j=1,\dots, t$. Here, $s$ and $t$ are parameters denoting the number of symmetric and CS pairs of asymmetric rank-1 tensors, respectively. The length of a CS PD equals $r=s+3t$. The matrices
$A := \left[ a_{1} \cdots a_{s} \right]$, 
$B := \left[ b_{1} \cdots b_{t} \right]$, $C := \left[ c_{1} \cdots c_{t} \right]$, and $D := \left[ d_{1} \cdots d_{t} \right]$, can be defined thus that the factor matrices can be written as: $U= \begin{bmatrix} A & B &  C & D \end{bmatrix}$, $V= \begin{bmatrix} A & D & B & C \end{bmatrix}$, and $W = \begin{bmatrix} A & C & D & B \end{bmatrix}$.
The number of unknowns is reduced by a factor 3, which makes this CS structure very useful for larger $m$. Furthermore, the cost function can be changed to
\begin{align} \label{eq:obj_CS}
\min_{A,B,C,D} \frac{1}{2} \Bigg( &\sum_{i=1}^{m^2} \sum_{j=i}^{m^2} \sum_{k=i+1}^{m^2} \bigg( T_{mmm}(i,j,k) - \sum_{i'=1}^s a_{ii'} a_{ji'} a_{ki'} \nonumber \\
- &\sum_{j'=1}^t \big( b_{ij'} d_{jj'} c_{kj'} + c_{ij'} b_{jj'} d_{kj'} + d_{ij'} c_{jj'} b_{kj'} \big) \bigg)^2 \\
+ &\sum_{i=1}^{m^2} \bigg( T_{mmm}(i,i,i) - \sum_{i'=1}^s a_{ii'}^3 - 3 \sum_{j'=1}^t \big( b_{ij'} d_{ij'} c_{ij'} \big) \bigg)^2 \Bigg), \nonumber
\end{align}
where $a_{ii'}:=A(i,i')$, $b_{ii'}:=B(i,i')$ and similarly for $c_{ii'}$ and $d_{ii'}$, which reduces the number of rows in the Jacobian matrix used in NLS optimization methods to solve \eqref{eq:obj_CS} from $m^6$ to $\frac{1}{3} (m^6 - m^2) + m^2$. Note that this cost function only includes one element for each CS pair of points because the CS structure already ensures that these elements are equal.
A disadvantage of including this structure may be that by restricting the search space we might not be able to find the most sparse or stable PDs of $T_{mmm}$. Furthermore, it is not known if the rank of a PD of $T_{mmm}$ with this structure equals the canonical rank.

Different practical CS PDs of rank 7 for $T_{222}$ and of rank 23 for $T_{333}$ are known in the literature \cite{Chiantini2019,Ballard2019}. These PDs were found using an ALS method and further investigated using algebraic geometry and group theory. Strassen's decomposition satisfies $s=1$ and $t=2$.
Although all $\mathrm{PD}_7\left( T_{222} \right)$s are known to be inv-equivalent, $\mathrm{PD}_{\mathrm{Strassen}}$ can be transformed into a practical PD with CS parameters $(s,t)=(4,1)$. However, this PD contains more nonzeros than $\mathrm{PD}_{\mathrm{Strassen}}$ and thus is not used in practice. 

\subsection{Decompositions obtained by recursion}

As mentioned in the introduction, an FMM algorithm applies a base algorithm for small $m$, $p$, and $n$, to large matrices, e.g., of size $m^k \times p^k$ and $p^k \times n^k$. In the same way, we can also obtain PDs of rank 49 of $T_{444}$ using a $\mathrm{PD}_7\left(T_{222} \right)$ one time recursively. The $\mathrm{PD}_{49}\left(T_{444} \right)$ that is obtained using $\mathrm{PD}_{\mathrm{Strassen}}$ is called $\mathrm{PD}_{\mathrm{Strassen}}^{\mathrm{rec}}$ in the rest of the text. In the following proposition, the formulas to obtain the factor matrices of a recursive PD from the factor matrices of the original PD are given. This result is well known, e.g., from the supplementary material of \cite{DeepMind_2022}, and in \cite{fam_FMM_2017}, and a proof can be found, e.g., in \cite[Section 2.2.3]{thesis}.

\begin{proposition}[Recursive PD] \label{prop:Strassen_rec}
If $U$, $V$, and $W$ are the factor matrices of a $\mathrm{PD}_r(T_{mpn})$, then the factor matrices $U', V'$, and $W'$, where
\begin{align} \label{eq:factor_matrices_rec}
u'_{i'} &:= \mathrm{vec} \left( U_{i_1} \otimes U_{i_2} \right), & v'_{i'} &:= \mathrm{vec} \left( V_{i_1} \otimes V_{i_2} \right), \\
w'_{i'} &:= \mathrm{vec} \left( W_{i_1} \otimes W_{i_2} \right), & i'&:=i_2+(i_1-1)r,
\end{align}
where `$\otimes$' denotes the Kronecker product, for $i_1,i_2 = 1, \dots, r$, are the factor matrices of a $\mathrm{PD}_{r^2}\left(T_{m^2p^2n^2} \right)$.
\end{proposition}

The following corollary gives the CS parameters and factor matrices of a recursive PD as a function of the original PD. 
 
\begin{corollary} \label{cor:CS_rec} If $U$, $V$, and $W$ are CS factor matrices of a $\mathrm{PD}_r\left( T_{mpn} \right)$ with CS parameters $s$ and $t$, then the factor matrices $U'$, $V'$, and $W'$ of the recursive decomposition $\mathrm{PD}_{r^2}^{\mathrm{rec}}\left( T_{m^2p^2n^2} \right)$ is also CS with parameters $s'=s^2$ and $t'=t(s+r)$: 
\begin{align*}
U' &:= \begin{bmatrix} A'~ B'~ C'~ D' \end{bmatrix}, & V' &:= \begin{bmatrix} A'~ D' ~B'~C' \end{bmatrix}, & W' &:= \begin{bmatrix} A'~C'~D'~B' \end{bmatrix},
\end{align*}
where 
\begin{align} \label{eq:CS_rec_A}
A'_{i'} &:= A_{i_1} \otimes A_{i_2}, & i' &:= i_2 + (i_1-1)s, & i_1,i_2 &= 1, \dots, s,
\end{align}
and 
\begin{align} \label{eq:CS_rec_1}
B' &:= \begin{bmatrix} B_1' & B_2' \end{bmatrix}, & C' &:= \begin{bmatrix}
C_1' & C_2' \end{bmatrix}, & D' &:= \begin{bmatrix} D_1' & D_2' \end{bmatrix},
\end{align}
where 
\begin{align} \label{eq:CS_rec_2}
B'_{1,j'} &:= A_{i_1} \otimes B_{j}, & C'_{1,j'} &:= A_{i_1} \otimes C_{j}, & D'_{1,j'} &:= A_{i_1} \otimes D_{j}, \\
B'_{2,k'} &:= B_{j} \otimes U_{k}, & C'_{2,k'} &:= C_{j} \otimes W_{k}, & D'_{2,k'} &:= D_{j} \otimes V_{k}, \nonumber
\end{align}
where $j' := j + (i_1-1)t$ and $k':=k + (j-1)r$, for $j = 1, \dots, t$ and $k = 1 ,\dots, r$. 
\end{corollary}
\begin{proof}
Using \Cref{prop:Strassen_rec}, we know that $U_{i'}' = U_{i_1} \otimes U_{i_2}$, $V_{i'}' = V_{i_1} \otimes V_{i_2}$, and $W_{i'}' = W_{i_1} \otimes W_{i_2}$, where $i':= i_2 + (i_1-1)r$. Thus, the symmetric part of the recursive PD must satisfy: $U_{i_1} = V_{i_1} = W_{i_1}$ and $U_{i_2} = V_{i_2}  = W_{i_2}$. Consequently, $i_1$ and $i_2$ must be smaller than or equal to $s$, such that $U_{i_1} = V_{i_1} = W_{i_1} = A_{i_1}$ and $U_{i_2} = V_{i_2}  = W_{i_2} = A_{i_2}$. 
For the other values of $i_1$ and $i_2$, the columns still appear in CS pairs and they can be rearranged as in \eqref{eq:CS_rec_2} to satisfy the CS structure. In \cite[Appendix A]{thesis}, the different columns that are obtained by using \Cref{prop:Strassen_rec} are written out to see this more clearly. 
\end{proof}

Note that, since $\mathrm{PD}_{\mathrm{Strassen}}$ satisfies $s=1$ and $t=2$, the parameters of $\mathrm{PD}_{\mathrm{Strassen}}^{\mathrm{rec}}$ are $s=1$ and $t=16$. Another CS $\mathrm{PD}_{49} \left( T_{444} \right)$ can be obtained with $(s,t) = (16,11)$ by using a $\mathrm{PD}_{7} \left( T_{222} \right)$ with $(s,t) = (4,1)$ one time recursively.

\section{Generalization cyclic symmetry} \label{sec:gen_CS}

In this section, a generalization of the CS structure discussed in \Cref{sec:CS}, to the case when $m$, $n$, and $p$ are not equal, is proposed. We do this to reduce the number of variables and reduce the search space to hopefully speed up the convergence. Note however that similarly as for the CS structure discussed in \Cref{sec:CS}, it is not known if the rank of PDs with this structure equals the canonical rank. Because a PD of $T_{pnm}$ or another permutation of $m$, $p$, and $n$, can be obtained from a PD of $T_{mpn}$, we assume in this chapter that $m \leq p \leq n$. It is known from \eqref{eq:ones_MMT} that:
\begin{equation*}
\begin{split}
&T_{mpn}(i_1 + (i_2-1)m,j_1 + (j_2-1)p,k_1 + (k_2-1)n) \\
&= \sum_{i=1}^r U_i(i_1,i_2) V_i(j_1,j_2) W_i(k_1,k_2) = \begin{cases} 1 & \mathrm{ if }~ i_1 = k_2, i_2 = j_1,j_2 = k_1 \\
0 & \mathrm{else}
\end{cases},
\end{split}
\end{equation*}
for all $i_1,k_2=1,\dots, m$, $i_2,j_1=1,\dots, p$, and $j_2,k_1=1,\dots, n$. 

Consequently, as long as $i_1,i_2,j_1,j_2,k_1,k_2 \leq m$, it holds that
\begin{equation*}
\begin{split}
&T_{mpn}(i_1 + (i_2-1)m,j_1 + (j_2-1)p,k_1 + (k_2-1)n) \\
=~ &T_{mpn}(k_1 + (k_2-1)m,i_1 + (i_2-1)p,j_1 + (j_2-1)n) \\
=~ &T_{mpn}(j_1 + (j_2-1)m,k_1 + (k_2-1)p,k_1 + (k_2-1)n).
\end{split}
\end{equation*}
And thus,
\begin{equation*}
\begin{split}
\sum_{i=1}^r U_i(i_1,i_2) V_i(j_1,j_2) W_i(k_1,k_2) &= \sum_{i=1}^r U_i(k_1,k_2) V_i(i_1,i_2) W_i(j_1,j_2)\\
&= \sum_{i=1}^r U_i(j_1,j_2) V_i(k_1,k_2) W_i(i_1,i_2),
\end{split}
\end{equation*}
which can be used to include a CS structure in a subpart of the factor matrices. An illustration is shown in \Cref{fig:tikz_gen_CS}, where the following submatrices are defined in function of the position in the factor matrices:
\begin{align*}
U_i(i_1,i_2) =: \begin{cases}
A_i(i_1,i_2) & \text{if}~ i_2 \leq m, i \leq s\\
B_{i-s}(i_1,i_2) & \text{if}~ i_2 \leq m, s < i \leq s+t\\
C_{i-s-t}(i_1,i_2) & \text{if}~ i_2 \leq m, s+t < i \leq s+2t\\
D_{i-s-2t}(i_1,i_2) & \text{if}~ i_2 \leq m, s+2t < i \leq r_{\mathrm{CS}}\\
\tilde{U}_i(i_1,i_2-m) & \text{if}~ i_2 > m,  i \leq r_{\mathrm{CS}} \\
\dot{U}_{i-r_{\mathrm{CS}}}(i_1,i_2) & \text{if}~ r_{\mathrm{CS}} < i \\
\end{cases},
\end{align*}
where $r_{\mathrm{CS}}:=s +3t$, and for $i=1,\dots, r$, $i_1=1,\dots, m$, and $i_2 = 1,\dots, p$, 
\begin{align*}
V_i(j_1,j_2) =: \begin{cases}
A_i(j_1,j_2) & \text{if}~ j_1,j_2 \leq m, i \leq s\\
D_{i-s}(j_1,j_2) & \text{if}~ j_1,j_2 \leq m, s < i \leq s+t\\
B_{i-s-t}(j_1,j_2) & \text{if}~ j_1,j_2 \leq m, s+t < i \leq s+2t\\
C_{i-s-2t}(j_1,j_2) & \text{if}~ j_1,j_2 \leq m, s+2t < i \leq r_{\mathrm{CS}}\\
\hat{V}_i(j_1-m,j_2) & \text{if}~ j_1 > m,j_2 \leq m,  i \leq r_{\mathrm{CS}} \\
\tilde{V}_i(j_1,j_2-m) & \text{if}~ j_2 > m,  i \leq r_{\mathrm{CS}} \\
\dot{V}_{i-r_{\mathrm{CS}}}(j_1,j_2) & \text{if}~ r_{\mathrm{CS}} < i \\
\end{cases},
\end{align*}
for $j_1=1,\dots, p$, and $j_2 = 1,\dots, n$, and
\begin{align*}
W_i(k_1,k_2) =: \begin{cases}
A_i(k_1,k_2) & \text{if}~ k_1 \leq m, i \leq s\\
B_{i-s}(k_1,k_2) & \text{if}~ k_1 \leq m, s < i \leq s+t\\
C_{i-s-t}(k_1,k_2) & \text{if}~ k_1 \leq m, s+t < i \leq s+2t\\
D_{i-s-2t}(k_1,k_2) & \text{if}~ k_1 \leq m, s+2t < i \leq r_{\mathrm{CS}}\\
\hat{W}_i(k_1-m,k_2) & \text{if}~ k_1 > m,  i \leq r_{\mathrm{CS}} \\
\dot{W}_{i-r_{\mathrm{CS}}}(k_1,k_2) & \text{if}~ r_{\mathrm{CS}} < i \\
\end{cases},
\end{align*}
for $k_1=1,\dots, n$, and $k_2 = 1,\dots, m$, where $A_i := \mathrm{reshape} \left(a_i,m \times m \right)$, for $i=1, \dots, s$, $B_j:= \mathrm{reshape} \left(b_j,m \times m \right)$, for $j=1, \dots, t$, and similarly for $C_j$ and $D_j$, and 
\begin{align*}
\tilde{U}_k &:= \mathrm{reshape} \big( \tilde{U}(:,k),~m \times (p-m) \big), \\
\tilde{V}_k &:= \mathrm{reshape} \big( \tilde{V}(:,k),~ p \times (n-m) \big),\\
\hat{V}_k &:= \mathrm{reshape} \big( \hat{V}(:,k),~ (p-m) \times m \big), \\
\hat{W}_k &:= \mathrm{reshape} \big( \hat{W}(:,k),~ (n-m) \times m \big),
\end{align*}
for $k=1,\dots, r_{\mathrm{CS}}$, and
\begin{align*}
\dot{U}_{i'} &:= \mathrm{reshape} \big( \dot{U}(:,i'),~m \times p \big), \\
\dot{V}_{i'} &:= \mathrm{reshape} \big( \dot{V}(:,i'),~ p \times n \big),\\
\dot{W}_{i'} &:= \mathrm{reshape} \big( \dot{W}(:,i'),~ n \times m \big),
\end{align*}
for $i'=1,\dots, (r-r_{\mathrm{CS}})$, and thus $\tilde{U} \in \mathbb{R}^{m(p-m) \times r_{\mathrm{CS}}}$, $\tilde{V} \in \mathbb{R}^{p(n-m) \times r_{\mathrm{CS}}}$, $\hat{V} \in \mathbb{R}^{(p-m)m \times r_{\mathrm{CS}}}$, $\hat{W} \in \mathbb{R}^{(n-m)m \times r_{\mathrm{CS}}}$, $\dot{U} \in \mathbb{R}^{mp \times (r-r_{\mathrm{CS}})}$, $\dot{V} \in \mathbb{R}^{pn \times (r-r_{\mathrm{CS}})}$, and $\dot{W} \in \mathbb{R}^{nm \times (r-r_{\mathrm{CS})}}$.
The proportions in \Cref{fig:tikz_gen_CS} are those of $T_{234}$, $(s,t):=(2,2)$, and ${r}:=20$. We show in the experiments that such PDs indeed exist. Note that, contrary to the case when $m$, $n$, and $p$ are equal, the CS rank $r_{\mathrm{CS}}$ does not equal the rank $r$, because this would be too restrictive. Also in the experiments that follow we did not find PDs for which both are equal. 
This is likely because of the interaction of the CS part with the parts $\tilde{U}$, $\tilde{V}$, $\hat{V}$, and $\hat{W}$. However, including this structure still reduces the number of parameters from $(mp+np+mn)r$ to $(mp+np+mn)r-2m^2 r_{\mathrm{CS}}$.

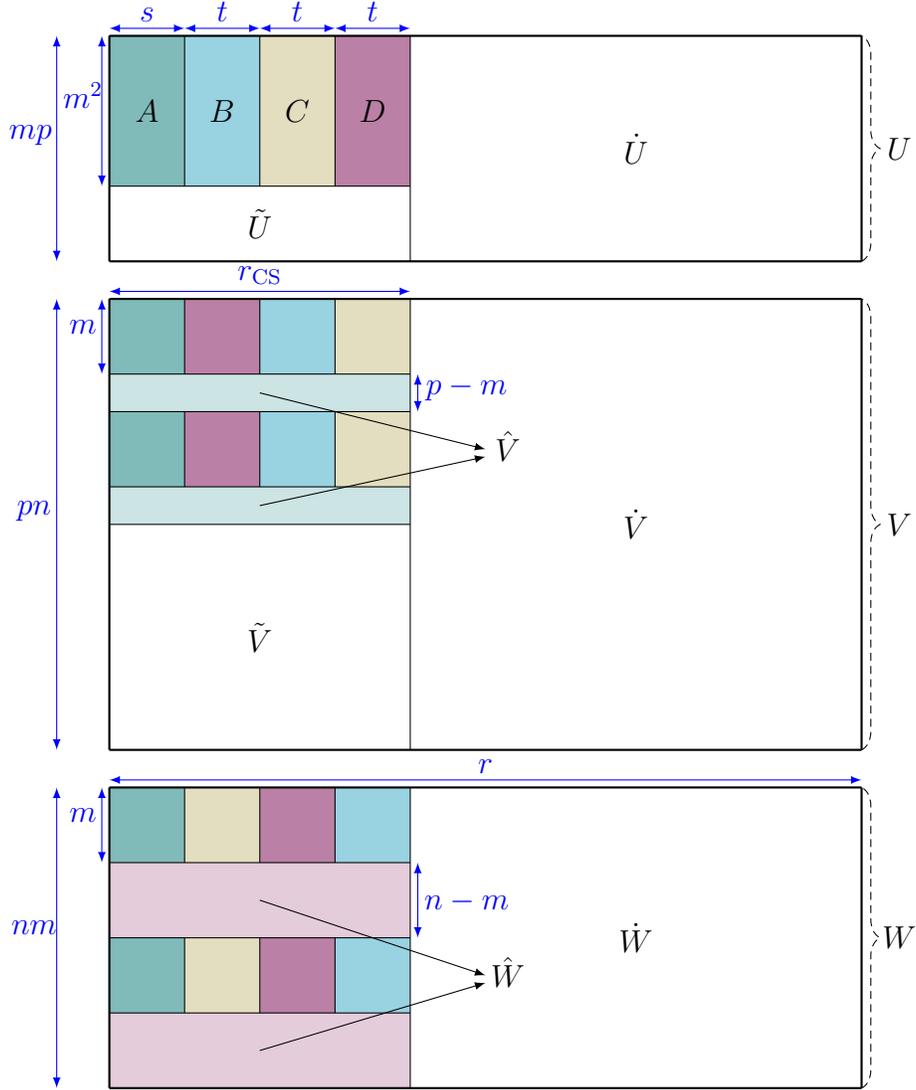
\begin{figure}[H]
\centering
\caption[Illustration of the generalized CS structure in the factor matrices.]{Illustration of the generalized CS factor matrices of a $\mathrm{PD}_{20}\left(T_{234}\right)$, with $(s,t):=(2,2)$.}
\begin{tikzpicture}
\coordinate (U1) at (0cm,3cm); 
\coordinate (U2) at (0cm,0cm); 
\coordinate (U3) at (10cm,0cm); 
\coordinate (U4) at (10cm,3cm); 
\coordinate (A2) at (0cm,1cm); 
\coordinate (A3) at (1cm,1cm); 
\coordinate (A4) at (1cm,3cm); 
\coordinate (B1) at (2cm,3cm); 
\coordinate (B2) at (2cm,1cm); 
\coordinate (C1) at (3cm,3cm); 
\coordinate (C2) at (3cm,1cm); 
\coordinate (D1) at (4cm,3cm); 
\coordinate (D2) at (4cm,1cm); 
\fill[mpigreen!50] (U1) -- (A2) -- (A3) -- (A4) -- cycle;
\fill[mpiblue!50] (B1) -- (B2) -- (A3) -- (A4) -- cycle;
\fill[mpisand!150] (C1) -- (C2) -- (B2) -- (B1) -- cycle;
\fill[mpired!50] (D1) -- (D2) -- (C2) -- (C1) -- cycle;
\draw[blue,<->] (-0.7cm,3cm) -- (-0.7cm,0cm) node[midway,above,xshift=-0.35cm,yshift=-0.1cm]{${mp}$};
\draw[blue,<->] (-0.1cm,3cm) -- (-0.1cm,1cm) node[midway,above,xshift=-0.25cm,yshift=-0.1cm]{${m^2}$};
\draw[blue,<->] (0cm,3.1cm) -- (1cm,3.1cm) node[midway,above,xshift=0cm,yshift=-0.05cm]{${s}$};
\draw[blue,<->] (1cm,3.1cm) -- (2cm,3.1cm) node[midway,above,xshift=0cm,yshift=-0.05cm]{${t}$};
\draw[blue,<->] (2cm,3.1cm) -- (3cm,3.1cm) node[midway,above,xshift=0cm,yshift=-0.05cm]{${t}$};
\draw[blue,<->] (3cm,3.1cm) -- (4cm,3.1cm) node[midway,above,xshift=0cm,yshift=-0.05cm]{${t}$};
\draw(0.5cm,2cm) node {$A$};
\draw(1.5cm,2cm) node {$B$};
\draw(2.5cm,2cm) node {$C$};
\draw(3.5cm,2cm) node {$D$};
\draw[thick] (U1) -- (U2); 
\draw[thick] (U2) -- (U3); 
\draw[thick] (U3) -- (U4); 
\draw[thick] (U4) -- (U1); 
\coordinate (U5) at (4cm,0cm); 
\draw (A2) -- (D2); 
\draw (U5) -- (D1);
\draw(2cm,0.5cm) node {$\tilde{U}$};
\draw(7cm,1.5cm) node {$\dot{U}$}; 
\draw (A3) -- (A4); 
\draw (B1) -- (B2);
\draw (C1) -- (C2); 

\draw[densely dashed,decorate,decoration={brace,amplitude=7pt,mirror},xshift=0cm,yshift=0cm] 
(U3) -- (U4) node[midway,xshift=0.5cm, yshift = 0cm] {$U$};

\coordinate (V1) at (0cm,-0.5cm); 
\coordinate (V2) at (0cm,-6.5cm); 
\coordinate (V3) at (10cm,-6.5cm); 
\coordinate (V4) at (10cm,-0.5cm); 
\coordinate (A2) at (0cm,-1.5cm); 
\coordinate (A3) at (1cm,-1.5cm); 
\coordinate (A4) at (1cm,-0.5cm); 
\fill[mpigreen!50] (V1) -- (A2) -- (A3) -- (A4) -- cycle;
\coordinate (D1) at (2cm,-0.5cm); 
\coordinate (D2) at (2cm,-1.5cm); 
\fill[mpired!50] (D1) -- (D2) -- (A3) -- (A4) -- cycle;
\coordinate (B1) at (3cm,-0.5cm); 
\coordinate (B2) at (3cm,-1.5cm); 
\fill[mpiblue!50] (D1) -- (D2) -- (B2) -- (B1) -- cycle;
\coordinate (C1) at (4cm,-0.5cm); 
\coordinate (C2) at (4cm,-1.5cm); 
\fill[mpisand!150] (C1) -- (C2) -- (B2) -- (B1) -- cycle;
\coordinate (V5) at (4cm,-6.5cm); 
\fill[mpigreen!20] (A2) -- (C2) -- (4cm,-2cm) -- (0cm,-2cm) -- cycle;
\draw (A2) -- (C2); 
\draw (A3) -- (A4); 
\draw (B1) -- (B2);
\draw (D1) -- (D2); 

\draw(7cm,-3.5cm) node {$\dot{V}$}; 

\coordinate (A1) at (0cm,-2cm); 
\coordinate (A2) at (0cm,-3cm); 
\coordinate (A3) at (1cm,-3cm); 
\coordinate (A4) at (1cm,-2cm); 
\fill[mpigreen!50] (A1) -- (A2) -- (A3) -- (A4) -- cycle;
\coordinate (D1) at (2cm,-2cm); 
\coordinate (D2) at (2cm,-3cm); 
\fill[mpired!50] (D1) -- (D2) -- (A3) -- (A4) -- cycle;
\coordinate (B1) at (3cm,-2cm); 
\coordinate (B2) at (3cm,-3cm); 
\fill[mpiblue!50] (D1) -- (D2) -- (B2) -- (B1) -- cycle;
\coordinate (C1) at (4cm,-2cm); 
\coordinate (C2) at (4cm,-3cm); 
\fill[mpisand!150] (C1) -- (C2) -- (B2) -- (B1) -- cycle;
\fill[mpigreen!20] (A2) -- (C2) -- (4cm,-3.5cm) -- (0cm,-3.5cm) -- cycle;
\draw (A3) -- (A4); 
\draw (B1) -- (B2);
\draw (D1) -- (D2); 

\draw[blue,<->] (-0.7cm,-0.5cm) -- (-0.7cm,-6.5cm) node[midway,above,xshift=-0.3cm,yshift=-0.1cm]{${pn}$};
\draw[blue,<->] (-0.1cm,-0.5cm) -- (-0.1cm,-1.5cm) node[midway,above,xshift=-0.25cm,yshift=-0.1cm]{${m}$};
\draw[blue,<->] (4.1cm,-1.5cm) -- (4.1cm,-2cm) node[midway,above,xshift=0.65cm,yshift=-0.25cm]{${p-m}$};
\draw[blue,<->] (0cm,-0.4cm) -- (4cm,-0.4cm) node[midway,above,xshift=0cm,yshift=-0.05cm]{${r}_{\mathrm{CS}}$};

\draw[thick] (V1) -- (V2); 
\draw[thick] (V2) -- (V3); 
\draw[thick] (V3) -- (V4); 
\draw[thick] (V4) -- (V1);
\draw (A2) -- (C2);
\draw (A1) -- (C1);
\draw (0cm,-3.5cm) -- (4cm,-3.5cm);
\coordinate (C1) at (4cm,-0.5cm); 
\draw (V5) -- (C1);
\draw(2cm,-5cm) node {$\tilde{V}$};
\draw[->] (2cm,-1.75cm) -- (5cm,-2.5cm) node[above,xshift=0.3cm,yshift=-0.3cm]{$\hat{V}$};
\draw[->] (2cm,-3.25cm) -- (5cm,-2.6cm);

\draw [densely dashed,decorate,decoration={brace,amplitude=7pt,mirror},xshift=0cm,yshift=0cm] 
(V3) -- (V4) node[midway,xshift=0.5cm, yshift = 0cm] {$V$};

\coordinate (W1) at (0cm,-7cm); 
\coordinate (W2) at (0cm,-11cm); 
\coordinate (W3) at (10cm,-11cm); 
\coordinate (W4) at (10cm,-7cm); 
\coordinate (A2) at (0cm,-8cm); 
\coordinate (A3) at (1cm,-8cm); 
\coordinate (A4) at (1cm,-7cm); 
\fill[mpigreen!50] (W1) -- (A2) -- (A3) -- (A4) -- cycle;
\coordinate (D1) at (2cm,-7cm); 
\coordinate (D2) at (2cm,-8cm); 
\fill[mpisand!150] (D1) -- (D2) -- (A3) -- (A4) -- cycle;
\coordinate (B1) at (3cm,-7cm); 
\coordinate (B2) at (3cm,-8cm); 
\fill[mpired!50] (D1) -- (D2) -- (B2) -- (B1) -- cycle;
\coordinate (C1) at (4cm,-7cm); 
\coordinate (C2) at (4cm,-8cm); 
\fill[mpiblue!50] (C1) -- (C2) -- (B2) -- (B1) -- cycle;
\fill[mpired!20] (A2) -- (C2) -- (4cm,-9cm) -- (0cm,-9cm) -- cycle;
\draw (A3) -- (A4); 
\draw (B1) -- (B2);
\draw (D1) -- (D2); 

\coordinate (A1) at (0cm,-9cm); 
\coordinate (A2) at (0cm,-10cm); 
\coordinate (A3) at (1cm,-10cm); 
\coordinate (A4) at (1cm,-9cm); 
\fill[mpigreen!50] (A1) -- (A2) -- (A3) -- (A4) -- cycle;
\coordinate (D1) at (2cm,-9cm); 
\coordinate (D2) at (2cm,-10cm); 
\fill[mpisand!150] (D1) -- (D2) -- (A3) -- (A4) -- cycle;
\coordinate (B1) at (3cm,-9cm); 
\coordinate (B2) at (3cm,-10cm); 
\fill[mpired!50] (D1) -- (D2) -- (B2) -- (B1) -- cycle;
\coordinate (C1) at (4cm,-9cm); 
\coordinate (C2) at (4cm,-10cm); 
\fill[mpiblue!50] (C1) -- (C2) -- (B2) -- (B1) -- cycle;
\fill[mpired!20] (A2) -- (C2) -- (4cm,-11cm) -- (W2) -- cycle;
\draw (A3) -- (A4); 
\draw (B1) -- (B2);
\draw (D1) -- (D2); 

\draw[blue,<->] (-0.7cm,-7cm) -- (-0.7cm,-11cm) node[midway,above,xshift=-0.3cm,yshift=-0.1cm]{${nm}$};
\draw[blue,<->] (-0.1cm,-7cm) -- (-0.1cm,-8cm) node[midway,above,xshift=-0.25cm,yshift=-0.1cm]{${m}$};
\draw[blue,<->] (4.1cm,-8cm) -- (4.1cm,-9cm) node[midway,above,xshift=0.65cm,yshift=-0.3cm]{${n-m}$};
\draw[blue,<->] (0cm,-6.9cm) -- (10cm,-6.9cm) node[midway,above,xshift=0cm,yshift=-0.05cm]{${r}$};
\draw[thick] (W1) -- (W2); 
\draw[thick] (W2) -- (W3); 
\draw[thick] (W3) -- (W4); 
\draw[thick] (W4) -- (W1); 
\draw (0cm,-8cm) -- (4cm,-8cm);
\draw (0cm,-9cm) -- (4cm,-9cm);
\draw (0cm,-10cm) -- (4cm,-10cm);
\draw (4cm,-7cm) -- (4cm,-11cm);
\draw[->] (2cm,-8.5cm) -- (5cm,-9.5cm) node[above,xshift=0.3cm,yshift=-0.3cm]{$\hat{W}$};
\draw[->] (2cm,-10.5cm) -- (5cm,-9.6cm);

\draw(7cm,-9cm) node {$\dot{W}$};

\draw [densely dashed,decorate,decoration={brace,amplitude=7pt,mirror},xshift=0cm,yshift=0cm] 
(W3) -- (W4) node[midway,xshift=0.5cm, yshift = 0cm] {$W$};

\end{tikzpicture}
\label{fig:tikz_gen_CS}
\end{figure}

\section{Numerical experiments} \label{sec:num_exp}

This section shows the results that we obtained with the AL method from \cite{vermeylen2023stability} to find PDs of MMTs with the generalized CS structure for different combinations of $m$, $p$, $n$, $s$ and $t$. The rank $r$ is always chosen as the lowest rank $\tilde{r}$ for which a solution is known to exist in the literature. For the first two experiments, this is also known to be the lowest rank possible \cite{ALEKSEYEV_1985,Alekseev2013OnTE}. The AL method takes as input an upper and lower bound on the elements in the factor matrices: $l \leq x \leq u$.
To show the advantage of the new structure, we set $u:=-l:=1$ and generate 50 random starting points of size $10^{-2}$ with the built-in function $\mlin{randn}$ of \Matlab:
\begin{equation} \label{eq:x0}
x_0:= 10^{-2} \cdot \mlin{randn} \left((mp+np+mn)r-2m^2 r_{\mathrm{CS}} \times 1 \right).
\end{equation}
The number of inner iterations of the LM method was set to 50 and the number of outer iterations to 15. The tolerance on the gradient and on the bound constraint were both set to $10^{-13}$.

\subsection*{Experiment 1: $\mathbf{T_{223}, r:={r}^*=11}$}
\Cref{tab:results_gen_CS_223} shows the results that are obtained with the AL method from \cite{vermeylen2023stability} to find $\mathrm{PD}_{11}\left(T_{223} \right)$s for different combinations of $s$ and $t$. The second to last column (from left to right) indicates the number of exact solutions found for these 50 random starting points. With an exact solution, a solution for which the optimality conditions, i.e., the tolerance on the gradient and constraint, are met, and furthermore for which the cost function is smaller than $10^{-12}$. The last column in the table indicates the number of practical solutions that were obtained when starting from the exact numerical PDs and by adding the constraint $h_{\mathrm{discr}}$ from \eqref{eq:h_discr}, scaled with a factor $0.1$,
to the optimization problem. As can be seen, for many combinations, exact and practical PDs exist. An example practical $\mathrm{PD}_{11}\left(T_{223} \right)$ with $(s,t):=(2,2)$ is given by the following equations:
\begin{equation} \label{eq:223_R11_S2_T2_discr}
\begin{split}
A &= \begin{bNiceMatrix}[r,columns-width=0.4cm] 
0 & 1 \\
0 & 0 \\
0 & 0 \\
1 & 0
\end{bNiceMatrix},~
B = \begin{bNiceMatrix}[r,columns-width=0.4cm] 
 0 & 0\\
 0 & 0\\
-1 &-1\\
 0 & 0
\end{bNiceMatrix},~ C = \begin{bNiceMatrix}[r,columns-width=0.4cm] 
-1 & 0\\
 1 &-1\\
-1 & 0\\
 0 & 0\\
\end{bNiceMatrix},~ D = \begin{bNiceMatrix}[r,columns-width=0.4cm]
 0 & 0\\
 1 & 1\\
 0 &-1\\
 0 & 1\\
\end{bNiceMatrix},\\
\hat{V} &= \begin{bNiceMatrix}[r,columns-width=0.4cm]
0 & 0 & 0 & 0 & 0 & 0 & 1 &-1 \\
0 & 0 & 0 & 0 & 0 & 0 &-1 & 1
\end{bNiceMatrix},~\tilde{U} = [],~\tilde{V} = [],  \\
\hat{W} &= \begin{bNiceMatrix}[r,columns-width=0.4cm]
0 & 0 &-1 & 0 & 0 & 0 & 0 &-1 \\
0 & 0 &-1 & 0 & 0 & 0 & 0 &-1
\end{bNiceMatrix}, \\
\dot{U} &= \begin{bNiceMatrix}[r,columns-width=0.4cm]
  0 &  1 &  0 \\ 
  0 & -1 & -1 \\
 -1 &  1 &  0 \\
  1 & -1 & -1 \\
\end{bNiceMatrix},~\dot{W} = \begin{bNiceMatrix}[r,columns-width=0.4cm]
  0 &  0 &  0 \\
  0 &  0 &  1 \\
  0 &  0 &  0 \\ 
  0 &  0 &  0 \\
 -1 & -1 &  0 \\
  1 &  0 & -1 \\
\end{bNiceMatrix},~
\dot{V} = \begin{bNiceMatrix}[r,columns-width=0.4cm]
  0 &  0 &  0 \\
  0 &  0 &  0 \\
  0 & -1 &  1 \\ 
  1 &  0 &  0 \\
  0 &  0 &  0 \\
  1 &  0 &  1 \\
\end{bNiceMatrix}.
\end{split}
\end{equation}
Also the size and minimal, maximal, and number of different ranks of $J_{\mathrm{CS,gen}}$, i.e., the Jacobian matrix of the new generalized CS structure, are given at the solutions. For a list of all the different ranks, we refer to the detailed results of the experiment, which are publicly available \footnote{\url{https://github.com/CharlotteVermeylen/ALM_FMM_gen_CS}}. Note that the second dimension of the Jacobian matrix and the ranks decrease when ${r}_{\mathrm{CS}}$ increases. 

\begin{table}[H]
\centering
\caption[Number of exact (numerical and practical) PDs of rank 11 of $T_{223}$, and the size and different ranks of the Jacobian matrix for different combinations of the CS parameters $s$ and $t$ using the AL method from \cite{vermeylen2023stability} on 50 small random starting points.]{Number of exact (numerical and practical) $\mathrm{PD}_{11}\left(T_{223} \right)$s that are obtained for different combinations of $s$ and $t$ using the AL method from \cite{vermeylen2023stability} on 50 random starting points generated as in \eqref{eq:x0}. Also the size and the different ranks of the Jacobian matrix are given.}
\label{tab:results_gen_CS_223}
\begin{tabular}{c | c c | c c c c | c c } 
\multirow{2}{*}{\boldmath $r_{\mathrm{CS}}$} & \multirow{2}{*}{\boldmath ${s}$} & \multirow{2}{*}{\boldmath ${t}$} & \multirow{2}{*}{\boldmath $\mathrm{size}(J_{\mathrm{CS,gen}})$} & \multicolumn{3}{c|}{\boldmath $\mathrm{rank}\left( J_{\mathrm{CS,gen}}(x^*) \right)$} & \multirow{2}{*}{\textbf{\#} \boldmath$x^*$} & \multirow{2}{*}{\textbf{pract.}} \\
& & & & $\min$ & $\max$ & \# & & \\
\midrule
10 & 7 & 1 & $144 \times 96$  & $82$ & $82$ & 1 & 10 & 0 \\
\hline
9 & 6 & 1 & $144 \times 104$ & $88$ & $89$ & 2 & 32 & 0 \\
\hline
\multirow{2}{*}{8} & 2 & 2 & \multirow{2}{*}{$144 \times 112$} & $92$ & $97$ & 4 & 8 & 1 \\
& 5 & 1 & & $93$ & $98$ & 5 & 45 & 0 \\
\hline
\multirow{3}{*}{7} & 1 & 2 & \multirow{3}{*}{$144 \times 120$} & $85$ & $98$ & 7 & 21 & 10 \\
& 4 & 1 & & $87$ & $103$ & 10 & 46 & 16 \\
& 7 & 0 & & $101$ & $105$ & 2 & 15 & 0 \\
\hline
\multirow{3}{*}{6} & 0 & 2 & \multirow{3}{*}{$144 \times 128$} & $91$ & $108$ & 14 & 29 & 13 \\
& 3 & 1 & & $93$ & $107$ & 13 & 48 & 36 \\
& 6 & 0 & & $107$ & $111$ & 5 & 29 & 0 \\
\hline
\multirow{2}{*}{5} & 2 & 1 & \multirow{2}{*}{$144 \times 136$} & $99$ & $117$ & 13 & 38 & 25 \\
& 5 & 0 & & $112$ & $117$ & 6 & 44 & 0 \\
\hline
\multirow{2}{*}{4} & 1 & 1 & \multirow{2}{*}{$144 \times 144$}  & $105$ & $123$ & 18 & 43 & 30 \\
 & 4 & 0 & & $107$ & $123$ & 16 & 48 & 18 \\
\hline
\multirow{2}{*}{3} & 0 & 1 & \multirow{2}{*}{$144 \times 152$} & $111$ & $129$ & 16 & 39 & 18 \\
& 3 & 0 & & $113$ & $129$ & 15 & 49 & 29  \\
\hline
2 & 2 & 0 & $144 \times 160$ & $119$ & $133$ & 14 & 50 & 21 \\
\hline
1 & 1 & 0 & $144 \times 168$ & $123$ & $135$ & 11 & 49 & 14 \\
\hline
0 & 0 & 0 & $144 \times 176$ & $125$ & $135$ & 9 & 50 & 15 \\
\bottomrule
\end{tabular}
\end{table}

As can be seen from \Cref{tab:results_gen_CS_223}, the highest value for ${r}_{\mathrm{CS}}$ that is obtained is 10 and is obtained for $(s,t):=(7,1)$. It was not possible to transform this PD into a practical one using the constraint $h_{\mathrm{discr}}$. 

In the last row, as a reference the results for $(s,t):=(0,0)$, i.e., without the CS structure, are given. Note that the number of numerical solutions in general decreases when ${r}_{\mathrm{CS}}$ increases but, on the other hand, the number of practical PDs increases for some combinations of $s$ and $t$. 

In \Cref{fig:results_223_11R_7S_0T_rng1-50_10-2_randn_ALM_ul1_30it_15good}, the convergence of the cost function is shown for the experiment with $(s,t):=(7,0)$. As can be seen, 15 out of the 50 starting points converge quickly to a numerically exact PD in less than 100 iterations and the other starting points get stuck in swaps.

\begin{figure}[H]
\hspace{-0.7cm}
\includegraphics[width= 1.1\linewidth]{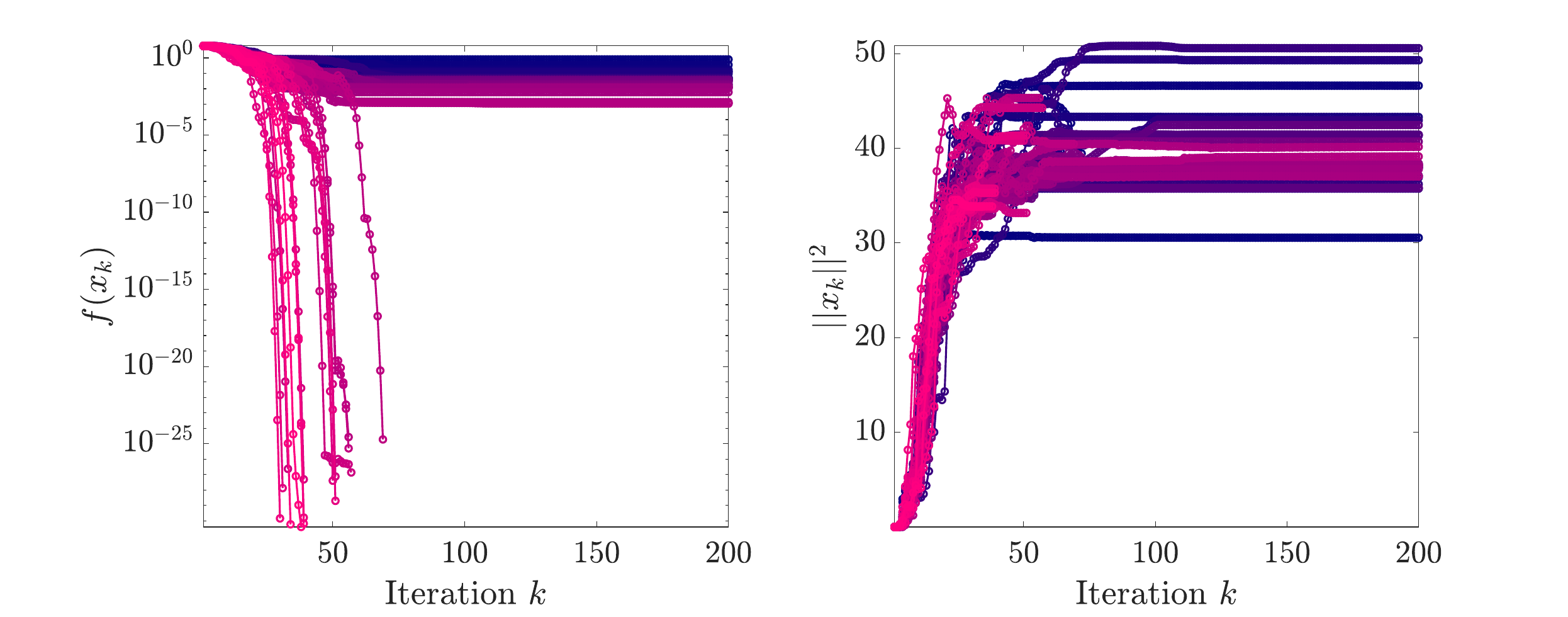}
\vspace{-0.5cm}
\caption[Convergence of the cost function to find PDs of rank 11 of $T_{223}$, with CS parameters $s=7$ and $t=0$ for 50 small random starting points using the AL method with a bound constraint of 1 on the absolute value of the elements in the factor matrices of the PD.]{Convergence of the cost function in \eqref{eq: unconstr_LSQ_min} using the AL method from \cite{vermeylen2023stability} with $u:=-l:=1$ and for 50 random starting points generated with $\mlin{randn}$ of size $10^{-2}$ to find $\mathrm{PD}_{11}\left(T_{223}\right)$s, with $(s,t):=(7,0)$.}
\label{fig:results_223_11R_7S_0T_rng1-50_10-2_randn_ALM_ul1_30it_15good}
\end{figure}

\subsection*{Experiment 2: $\mathbf{T_{224}, r:=r^*=14}$} The results for $\mathrm{PD}_{14}\left(T_{224} \right)$s are shown in \Cref{tab:results_gen_CS_224}, again for different values of $r_{\mathrm{CS}}$, $s$ and $t$. The largest CS rank that is obtained is again $r_{\mathrm{CS}} = 10$ but now both PDs with $(s,t)= (4,2)$ and $(s,t)= (7,1)$ were found. However, again none of these PDs were discretizable using the constraint $h_{\mathrm{discr}}$.

An example practical $\mathrm{PD}_{14}\left(T_{224} \right)$ that is obtained for $(s,t)= (3,1)$ has as submatrices:
\begin{align*} 
A &= \begin{bNiceMatrix}[r,columns-width=0.4cm] 
1 & 1 & 0\\
1 & 0 & 0\\
0 &-1 & 0\\
0 & 0 & 1
\end{bNiceMatrix},~
B = \begin{bNiceMatrix}[r,columns-width=0.4cm] 
-1 \\
-1 \\
 1 \\
 1 
\end{bNiceMatrix},~ C = \begin{bNiceMatrix}[r,columns-width=0.4cm] 
 0 \\
 0 \\
 1 \\
 0 
\end{bNiceMatrix},~ D = \begin{bNiceMatrix}[r,columns-width=0.4cm] 
 0 \\
 1 \\
 0 \\
 0 
\end{bNiceMatrix},\\ 
\tilde{V} &= \hat{W} = \begin{bNiceMatrix}[r,columns-width=0.4cm]
0 & 0 & 0 & 0 & 0 & 0 \\
0 & 0 & 0 & 0 & 0 & 0 \\
0 & 0 & 0 & 0 & 0 & 0 \\
0 & 0 & 0 & 0 & 0 & 0 
\end{bNiceMatrix},~\tilde{U} = [],~\hat{V} = [],\\
\dot{U} &= \begin{bNiceMatrix}[r,columns-width=0.4cm]
 0 & 0 & 1 & 1 & 1 & 1 & 0 &-1\\
-1 & 0 & 0 & 1 & 1 & 0 & 0 & 0\\
 0 &-1 & 0 &-1 & 0 & 1 & 0 & 0\\
-1 &-1 & 0 & 0 & 0 & 0 &-1 & 1
\end{bNiceMatrix},\\
\dot{V} &= \begin{bNiceMatrix}[r,columns-width=0.4cm]
0 & 0 & 0 &-1 & 0 & 0 & 0 & 0\\
0 & 0 & 0 &-1 & 0 & 0 & 0 & 0\\
0 & 0 & 0 & 1 & 0 & 0 & 0 & 0\\
0 & 0 & 0 & 0 & 0 & 0 & 0 & 0\\
0 & 0 &-1 & 0 &-1 & 0 & 0 & 0\\
0 &-1 & 1 & 0 & 0 &-1 & 0 & 1\\
1 & 0 & 0 & 0 & 1 & 0 & 1 &-1\\
0 & 1 & 0 & 0 & 0 & 0 &-1 & 0
\end{bNiceMatrix},\\
\dot{W} &= \begin{bNiceMatrix}[r,columns-width=0.4cm]
 0 & 0 & 0 & 1 & 0 & 0 & 0 & 0\\
 0 & 0 & 0 & 1 & 0 & 0 & 0 & 0\\
 0 & 0 &-1 & 0 & 0 &-1 & 0 & 0\\
 0 &-1 & 0 & 0 & 0 &-1 &-1 & 1\\
 0 & 0 & 0 &-1 & 0 & 0 & 0 & 0\\
 0 & 0 & 0 & 0 & 0 & 0 & 0 & 0\\
-1 & 0 & 1 & 0 &-1 & 0 & 0 & 1\\
-1 & 0 & 0 & 0 & 0 & 0 & 1 & 0
\end{bNiceMatrix}.
\end{align*}
Again it can be noticed that by including the generalized CS structure, more practical PDs are obtained and furthermore, more different values of the rank of the Jacobian and thus more non-inv-equivalent PDs are obtained as long as the CS rank is not set too high.

\begin{table}[H]
\centering
\caption[Number of exact (numerical and practical) PDs of rank 14 of $T_{224}$, and the size and different ranks of the Jacobian matrix for different combinations of the CS parameters $s$ and $t$ using the AL method from \cite{vermeylen2023stability} on 50 small random starting points.]{Number of exact (numerical and practical) $\mathrm{PD}_{14}\left(T_{224} \right)$s that are obtained for different combinations of $s$ and $t$ using the AL method from \cite{vermeylen2023stability} on 50 random starting points generated as in \eqref{eq:x0}. Also the size and the minimal, maximal, and number of different ranks of the Jacobian matrix are given.}
\label{tab:results_gen_CS_224}
\begin{tabular}{c | c c | c c c c | c c } 
\multirow{2}{*}{\boldmath $r_{\mathrm{CS}}$} & \multirow{2}{*}{\boldmath ${s}$} & \multirow{2}{*}{\boldmath ${t}$} & \multirow{2}{*}{\boldmath $\mathrm{size}(J_{\mathrm{CS,gen}})$} & \multicolumn{3}{c|}{\boldmath $\mathrm{rank}\left( J_{\mathrm{CS,gen}}(x^*) \right)$} & \multirow{2}{*}{\textbf{\#} \boldmath$x^*$} & \multirow{2}{*}{\textbf{pract.}} \\
& & & & $\min$ & $\max$ & \# & & \\
\midrule
\multirow{2}{*}{10} & 4 & 2 & \multirow{2}{*}{$256 \times 200$} & $170$ & 170 & 1 & 1 & 0\\
& 7 & 1 & & $172$ & $178$ & 2 & 6 & 0 \\
\hline
\multirow{2}{*}{9} & 3 & 2 & \multirow{2}{*}{$256 \times 208$} & $183$ & $183$ & 1 & 2 & 0\\
& 6 & 1 & & $178$ & $185$ & 5 & 15 & 0 \\
\hline
\multirow{3}{*}{8} & 2 & 2 & \multirow{3}{*}{$256 \times 216$} & $187$ & $190$ & 3 & 6 & 0\\
& 5 & 1 & & $185$ & $190$ & 5 & 38 & 0 \\
& 8 & 0 & & $184$ & $191$ & 2 & 3 & 0 \\
\hline
\multirow{4}{*}{7} & 1 & 2 & \multirow{4}{*}{$256 \times 224$} & $178$ & $193$ & 6 & 22 & 9 \\
& 4 & 1 & & $180$ & 197 & 10 & 44 & 14 \\
& 7 & 0 & & $190$ & $199$ & 5 & 20 & 0 \\
\hline
\multirow{3}{*}{6} & 0 & 2 & \multirow{3}{*}{$256 \times 232$} & $184$ & 198 & 9 &  20 & 9 \\
& 3 & 1 & & $184$ & $201$ & 10 & 41 & 21 \\
& 6 & 0 & & $196$ & $205$ & 7 & 20 & 0 \\
\hline
\multirow{2}{*}{5} & 2 & 1 & \multirow{2}{*}{$256 \times 240$} & $188$ & $212$ & 15 &  36 & 16 \\
& 5 & 0 & & $205$ & $210$ & 6 & 36 & 0 \\
\hline
\multirow{2}{*}{4} & 1 & 1 & \multirow{2}{*}{$256 \times 248$} & $196$ & $214$ & 13 & 30 & 12 \\
& 4 & 0 &  & $196$ & $217$ & 15 & 44 & 12 \\
\hline
\multirow{2}{*}{3} & 0 & 1 & \multirow{2}{*}{$256 \times 256$} & $202$ & $221$ & 11 & 25 & 4 \\
& 3 & 0 & & $202$ & $221$ & 15 & 42 & 11 \\
\hline
2 & 2 & 0 & $256 \times 264$ & $214$ & $227$ & 10 & 49 & 4 \\
\hline
1 & 1 & 0 & $256 \times 272$ & $217$ & $228$ & 9 & 45 & 4 \\
\hline
0 & 0 & 0 & $256 \times 280$ & $217$ & $228$ & 8 & 50 & 5 \\
\bottomrule
\end{tabular}
\end{table}

\subsection*{Experiment 3: $\mathbf{T_{233}, r:=\tilde{r}=15}$}
Thirdly, in \Cref{tab:results_gen_CS_233}, the results for $T_{233}$, rank 15 are shown. Now, the maximal value for $r_{\mathrm{CS}}$ that is obtained is 9, and one practical PD is obtained with parameters $(s,t)=(6,1)$:
\begin{equation} \label{eq:233_R15_S6_T1_discr}
\begin{split}
A &= \begin{bNiceMatrix}[r,columns-width=0.4cm] 
0 & 1 &  0 & 0 &  0 & 0  \\
0 & 0 & -1 & 0 &  1 & 0  \\
0 & 0 &  0 & 1 & -1 & 0 \\
0 & 0 &  1 & 1 & -1 & 0
\end{bNiceMatrix},~
B = \begin{bNiceMatrix}[r,columns-width=0.4cm] 
 0 \\
 0 \\
 1 \\
 0 
\end{bNiceMatrix},~ C = \begin{bNiceMatrix}[r,columns-width=0.4cm] 
  1 \\
  1 \\
 -1 \\
 -1
\end{bNiceMatrix},~ D = \begin{bNiceMatrix}[r,columns-width=0.4cm] 
  0 \\
  1 \\
  0 \\
  0
\end{bNiceMatrix}, \nonumber \\
\tilde{U} &= \begin{bNiceMatrix}[r,columns-width=0.4cm] 
 1 & 0 & 0 & 0 & 0 &  1 & 0 & 0 & 0 \\
 1 & 0 & 0 & 0 & 0 &  0 & 0 & 0 & 0
\end{bNiceMatrix}, \\
\hat{V} &= \begin{bNiceMatrix}[r,columns-width=0.4cm]
 0 & 0 & 0 & 0 &  0 & 1 &  0 &  0 &  1 \\
 0 & 1 & 0 & 1 & -1 &-1 &  1 &  1 & -1 
\end{bNiceMatrix}, \\
\tilde{V} &= \begin{bNiceMatrix}[r,columns-width=0.4cm]
-1 & -1 & 0 & 0 &  0 & 0 &  0 &  0 &  1 \\
 0 &  0 & 1 & 1 & -1 & 0 & -1 &  0 & -1 \\
-1 &  0 & 0 & 0 &  0 &-1 &  0 &  0 &  0
\end{bNiceMatrix}, \\
\hat{W} &= \begin{bNiceMatrix}[r,columns-width=0.4cm]
 0 & 0 & 0 & 0 & 0 & -1 & 0 & 0 & 0 \\
-1 & 0 & 0 & 0 & 0 & -1 & 0 & 0 & 0 \\
\end{bNiceMatrix},\\
\dot{U} &= \begin{bNiceMatrix}[r,columns-width=0.4cm]
  0 &  1 &  1 &  0 &  0 & 0 \\
 -1 &  0 &  1 &  1 &  1 & 1 \\ 
  0 &  1 &  0 &  0 & -1 & 0 \\
  0 &  0 &  0 &  0 &  0 & 1 \\
  1 & -1 & -1 &  0 &  0 & 0 \\ 
  1 &  0 & -1 & -1 & -1 &-1 \\ 
\end{bNiceMatrix},  \\
\dot{V} &= \begin{bNiceMatrix}[r,columns-width=0.4cm]
  0 &  0 &  0 &  0 &  0 &  0 \\
  0 &  0 &  0 &  0 &  0 &  0 \\
 -1 &  0 &  0 &  1 &  0 &  0 \\ 
  0 &  0 &  0 &  0 &  0 &  0 \\
  0 &  0 &  0 &  0 &  0 &  0 \\
  1 &  1 & -1 & -1 &  1 &  0 \\
 -1 &  0 &  1 &  0 &  0 &  0 \\ 
  0 &  0 &  0 &  1 & -1 &  1 \\
  0 &  0 &  0 &  0 &  0 &  0 \\ 
\end{bNiceMatrix},~\dot{W} = \begin{bNiceMatrix}[r,columns-width=0.4cm]
 -1 &  0 &  1 &  1 &  1 &  0 \\
  0 & -1 &  0 &  0 &  0 &  0 \\
 -1 &  1 &  1 &  1 &  1 &  0 \\ 
  0 &  0 &  0 & -1 & -1 &  0 \\
  0 &  0 &  0 &  0 & -1 & -1 \\
  1 &  0 &  0 & -1 &  0 &  1 \\
\end{bNiceMatrix}, 
\end{split}
\end{equation} 
 
\begin{table}[H]
\centering
\caption[Number of exact (numerical and practical) PDs of rank 15 of $T_{233}$, and the size and different ranks of the Jacobian matrix for different combinations of the CS parameters $s$ and $t$ using the AL method from \cite{vermeylen2023stability} on 50 small random starting points.]{Number of exact (numerical and practical) $\mathrm{PD}_{15}\left(T_{233} \right)$s that are obtained for different combinations of $s$ and $t$ using the AL method from \cite{vermeylen2023stability} on 50 random starting points generated as in \eqref{eq:x0}. Also the size and the different ranks of the Jacobian matrix are given.}
\label{tab:results_gen_CS_233}
\begin{tabular}{c | c c | c c c c | c c } 
\multirow{2}{*}{\boldmath $r_{\mathrm{CS}}$} & \multirow{2}{*}{\boldmath ${s}$} & \multirow{2}{*}{\boldmath ${t}$} & \multirow{2}{*}{\boldmath $\mathrm{size}(J_{\mathrm{CS,gen}})$} & \multicolumn{3}{c|}{\boldmath $\mathrm{rank}\left( J_{\mathrm{CS,gen}}(x^*) \right)$} & \multirow{2}{*}{\textbf{\#} \boldmath$x^*$} & \multirow{2}{*}{\textbf{pract.}} \\
& & & & $\min$ & $\max$ & \# & & \\
\midrule
\multirow{2}{*}{9} & 3 & 2 & \multirow{2}{*}{$324 \times 243$} & $216$ & $216$ & 1 & 2 & 0 \\
& 6 & 1 & & $218$ & $218$ & 1 & 1 & 1 \\
\hline
\multirow{2}{*}{8} & 2 & 2 & \multirow{2}{*}{$324 \times 251$} & $223$ & 227 & 2 & 2 & 0 \\
& 5 & 1 & & $225$ & $226$ & 2 & 4 & 2 \\
\hline
\multirow{2}{*}{7} & 1 & 2 & \multirow{2}{*}{$324 \times 259$}& $227$ & $228$ & 2 & 13 & 2 \\
& 4 & 1 & & $229$ & $234$ & 5 & 15 & 5 \\
\hline
\multirow{3}{*}{6} & 0 & 2 & \multirow{3}{*}{$324 \times 267$} & $232$ & $234$ & 3 & 10 & 1  \\ 
& 3 & 1 & & $234$ & $239$ & 6 & 25 & 10 \\
& 6 & 0 & & $237$ & 238 & 2 & 2 & 1  \\
\hline
\multirow{2}{*}{5} & 2 & 1 & \multirow{2}{*}{$324 \times 275$} & $240$ & $244$ & 4 & 12 & 7 \\
& 5 & 0 & & $244$ & $245$ & 2 & 2 & 0 \\
\hline
\multirow{2}{*}{4} & 1 & 1 & \multirow{2}{*}{$324 \times 283$} & $247$ & $250$ & 3 & 5 & 2 \\
& 4 & 0 & & $249$ & $254$ & 5 & 11 & 6 \\
\hline
\multirow{2}{*}{3} & 0 & 1 & \multirow{2}{*}{$324 \times 291$} & $252$ & $256$ & 3 & 6 & 2 \\
& 3 & 0 & & $254$ & $259$ & 6 & 18 & 6 \\
\hline
\multirow{1}{*}{2} & 2 & 0 & $324 \times 299$ & $254$ & $262$ & 6 & 20 & 4 \\
\hline
\multirow{1}{*}{1} & 1 & 0 & $324 \times 307$ & $260$ & $264$ & 5 & 30 & 0 \\
\hline
\multirow{1}{*}{0} & 0 & 0 & $324 \times 315$ & $262$ & $264$ & 3 & 29 & 0 \\
\bottomrule
\end{tabular}
\end{table}

\subsection*{Experiment 4: $\mathbf{T_{225}, r:=\tilde{r}=18}$}
Fourthly, in \Cref{tab:results_gen_CS_225}, the results for $T_{225}$, rank 18, are shown. The maximal CS rank is 10 as in \Cref{tab:results_gen_CS_223} and \Cref{tab:results_gen_CS_224}. However, the highest value of the CS rank for which a practical PD is obtained is 9 and is obtained for $(s,t) = (6,1)$. The CS submatrices of the PD that is obtained are:
\begin{equation*} 
\begin{split}
A &= \begin{bNiceMatrix}[r,columns-width=0.4cm] 
1 &-1 &  1 & 0 & 0 & 0  \\
0 & 1 & -1 & 0 & 0 & 0  \\
1 &-1 &  0 & 0 & 0 & 0 \\
0 & 0 &  0 & 1 & 1 &-1
\end{bNiceMatrix},~
B = \begin{bNiceMatrix}[r,columns-width=0.4cm] 
 0 \\
 1 \\
 0 \\
 0 
\end{bNiceMatrix},~ C = \begin{bNiceMatrix}[r,columns-width=0.4cm] 
  1 \\
 -1 \\
  1 \\
 -1
\end{bNiceMatrix},~ D = \begin{bNiceMatrix}[r,columns-width=0.4cm] 
  0 \\
  0 \\
 -1 \\
  0
\end{bNiceMatrix}, 
\end{split}
\end{equation*}
\begin{align*} 
\tilde{V} &= \begin{bNiceMatrix}[r,columns-width=0.4cm]
 0 & 0 & 0 & 0 & 0 & 0 & 0 & 0 & 0 \\
 0 & 0 & 0 &-1 & 0 & 1 & 0 & 0 & 0 \\ 
 0 & 0 & 0 & 0 & 0 & 0 & 0 & 0 & 0 \\
 0 & 0 & 0 &-1 & 0 & 1 & 0 & 0 & 0 \\
 0 & 0 & 0 & 0 & 0 & 0 & 0 & 0 & 0 \\ 
 0 & 0 & 0 & 0 &-1 & 1 & 0 & 0 & 0 \\
\end{bNiceMatrix},~ \tilde{U} = [],~ \hat{V} = [], 
\\
\hat{W} &= \begin{bNiceMatrix}[r,columns-width=0.4cm]
 0 & 0 & 0 & 0 & 0 & 0 & 0 & 0 & 0 \\
 0 & 0 & 0 & 0 &-1 & 1 & 0 & 0 & 0 \\ 
 0 & 0 & 0 &-1 & 0 & 1 & 0 & 0 & 0 \\
 0 & 0 & 0 & 0 & 0 & 0 & 0 & 0 & 0 \\
 0 & 0 & 0 & 0 & 1 &-1 & 0 & 0 & 0 \\ 
 0 & 0 & 0 & 1 & 0 &-1 & 0 & 0 & 0 \\
\end{bNiceMatrix},
\\
\dot{U} &= \begin{bNiceMatrix}[r,columns-width=0.4cm]
 0 & 1 & 0 & 1 & 0 & 0 &-1 &-1 & 1 \\
 0 & 0 & 0 & 0 & 1 & 1 & 0 & 0 & 1 \\ 
 1 & 0 & 1 &-1 & 0 & 0 & 1 & 0 &-1 \\
 1 & 0 & 1 &-1 & 0 & 0 & 0 & 0 &-1 \\
\end{bNiceMatrix}, 
\\
\dot{V} &= \begin{bNiceMatrix}[r,columns-width=0.4cm]
 0 & 0 & 0 & 0 & 0 & 0 & 0 & 0 & 0 \\
 0 & 0 & 0 & 0 & 0 & 0 & 0 & 0 & 0 \\ 
 0 & 0 & 0 & 0 & 0 & 0 & 0 & 0 & 0 \\
 0 & 0 & 0 & 0 & 0 & 0 & 0 & 0 & 0 \\
 0 & 1 & 0 & 0 &-1 & 0 & 0 & 0 & 0 \\ 
 0 & 1 & 0 &-1 & 0 & 0 &-1 & 0 & 0 \\
 0 & 0 & 1 & 1 &-1 & 0 & 0 & 0 & 1 \\
 0 & 0 & 1 & 0 & 0 & 0 & 0 & 0 & 0 \\ 
 0 & 0 & 1 & 1 & 0 &-1 & 0 &-1 & 1 \\
-1 & 0 & 1 & 0 & 0 & 0 & 0 & 0 & 0 \\
\end{bNiceMatrix},
\\
\dot{W} &= \begin{bNiceMatrix}[r,columns-width=0.4cm]
 0 & 0 & 0 & 0 & 0 & 0 & 0 & 0 & 0 \\
 0 & 0 & 0 & 0 & 0 & 0 & 0 & 0 & 0 \\ 
 0 & 1 & 0 & 0 & 0 & 0 &-1 & 0 & 0 \\
 1 & 0 & 1 & 1 & 0 & 0 & 1 &-1 & 0 \\
-1 & 0 & 0 & 0 & 0 & 0 & 0 & 1 & 0 \\ 
 0 & 0 & 0 & 0 & 0 & 0 & 0 & 0 & 0 \\
 0 & 0 & 0 & 0 & 0 & 0 & 0 & 0 & 0 \\
 0 & 0 & 0 & 1 &-1 &-1 & 1 & 0 &-1 \\ 
 0 & 0 & 0 &-1 & 0 & 1 &-1 & 0 & 1 \\
 0 & 0 & 0 & 0 & 0 &-1 & 0 & 0 & 0 \\
\end{bNiceMatrix}. 
\end{align*}

\begin{table}[H]
\centering
\caption[Number of exact (numerical and practical) PDs of rank 18 of $T_{225}$, and the size and different ranks of the Jacobian matrix for different combinations of the CS parameters $s$ and $t$ using the AL method from \cite{vermeylen2023stability} on 50 small random starting points.]{Number of exact (numerical and practical) $\mathrm{PD}_{18}\left(T_{225} \right)$s for different combinations of $s$ and $t$ using the AL method from \cite{vermeylen2023stability} on 50 random starting points generated as in \eqref{eq:x0}. Also the size and the different ranks of the Jacobian matrix are given.}
\label{tab:results_gen_CS_225}
\begin{tabular}{c | c c | c c c c | c c } 
\multirow{2}{*}{\boldmath $r_{\mathrm{CS}}$} & \multirow{2}{*}{\boldmath ${s}$} & \multirow{2}{*}{\boldmath ${t}$} & \multirow{2}{*}{\boldmath $\mathrm{size}(J_{\mathrm{CS,gen}})$} & \multicolumn{3}{c|}{\boldmath $\mathrm{rank}\left( J_{\mathrm{CS,gen}}(x^*) \right)$} & \multirow{2}{*}{\textbf{\#} \boldmath$x^*$} & \multirow{2}{*}{\textbf{pract.}} \\
& & & & $\min$ & $\max$ & \# & & \\
\midrule
\multirow{2}{*}{10} & 4 & 2 & \multirow{2}{*}{$400 \times 352$} & 302 & 315 & 9 & 13 & 0\\
& 7 & 1 & & 298 & 310 & 10 & 34 & 0 \\
\hline
\multirow{2}{*}{9} & 3 & 2 & \multirow{2}{*}{$400 \times 360$} & 309 & 319 & 9 & 17 & 0\\
& 6 & 1 & & 288 & 316 & 11 & 44 & 1 \\
\hline
\multirow{3}{*}{8} & 2 & 2 & \multirow{3}{*}{$400 \times 368$} & $308$ & $326$ & 11 & 19 & 0 \\
& 5 & 1 & & $313$ & $323$ & 9 & 49 & 0 \\
& 8 & 0 & & $308$ & $325$ & 12 & 23 & 0 \\
\hline
\multirow{3}{*}{7} & 1 & 2 & \multirow{3}{*}{$400 \times 376$} & $292$ & $331$ & 18 & 33 & 4 \\
& 4 & 1 & & $293$ & $327$ & 15 & 46 & 3 \\
& 7 & 0 & & $314$ & $332$ & 17 & 39 & 0 \\
\hline
\multirow{3}{*}{6} & 0 & 2 & \multirow{3}{*}{$400 \times 384$} & $297$ & $339$ & 14 & 23 & 3 \\
& 3 & 1 & & $299$ & $338$ & 23 & 46 & 10 \\
& 6 & 0 & & $319$ & $339$ & 19 & 41 & 0 \\
\hline
\multirow{2}{*}{5} & 2 & 1 & \multirow{2}{*}{$400 \times 392$} & 305 & 341 & 19 & 47 & 8 \\
& 5 & 0 & & $329$ & $341$ & 13 & 49 & 0 \\
\hline
\multirow{2}{*}{4} & 1 & 1 & \multirow{2}{*}{$400 \times 400$} & $313$ & 352 & 23 & 49 & 4 \\
& 4 & 0 & & $312$ & 351 & 21 & 48 & 5 \\
\hline
\multirow{2}{*}{3} & 0 & 1& \multirow{2}{*}{$400 \times 408$} & $327$ & $357$ & 20 & 46 & 1 \\
& 3 & 0 & & $322$ & $354$ & 21 & 48 & 5 \\
\hline
2 & 2 & 0 & $400 \times 416$ & $330$ & $354$ & 17 & 49 & 2 \\
\hline
1 & 1 & 0 & $400 \times 424$ & $332$ & $356$ & 16 & 49 & 2 \\
\hline
0 & 0 & 0 & $400 \times 432$ & $336$ & $358$ & 15 & 50 & 1 \\
\bottomrule
\end{tabular}
\end{table}

\subsection*{Experiment 5: $\mathbf{T_{234}, r:=\tilde{r}=20}$}
\Cref{tab:results_gen_CS_234} shows the results for $T_{234}$, rank 20. The parameters $s$ and $t$ for which solutions are found are similar to the previous experiments but the number of solutions is on average smaller because the problem becomes more difficult as $p$, $n$, and the rank $r$ increase.
Also, fewer different values of the rank of the Jacobian matrix are found. Note that the number of different values of the rank in general increases for moderate values of $r_{\mathrm{CS}}$. Furthermore, including the CS structure again helps to find practical solutions compared to the generic case ($(s,t)=(0,0)$) for different combinations of $s$ and $t$. 

\begin{table}[H]
\centering
\caption[Number of exact (numerical and practical) PDs of rank 20 of $T_{234}$, and the size and different ranks of the Jacobian matrix for different combinations of the CS parameters $s$ and $t$ using the AL method from \cite{vermeylen2023stability} on 50 small random starting points.]{Number of exact (numerical and practical) $\mathrm{PD}_{20}\left(T_{234} \right)$s that we obtained for different combinations of $(s,t)$ using the AL method from \cite{vermeylen2023stability} on 50 random starting points generated as in \eqref{eq:x0}. Also the size and the different ranks of the Jacobian matrix are given.}
\label{tab:results_gen_CS_234}
\begin{tabular}{c | c c | c c c c | c c } 
\multirow{2}{*}{\boldmath $r_{\mathrm{CS}}$} & \multirow{2}{*}{\boldmath ${s}$} & \multirow{2}{*}{\boldmath ${t}$} & \multirow{2}{*}{\boldmath $\mathrm{size}(J_{\mathrm{CS,gen}})$} & \multicolumn{3}{c|}{\boldmath $\mathrm{rank}\left( J_{\mathrm{CS,gen}}(x^*) \right)$} & \multirow{2}{*}{\textbf{\#} \boldmath$x^*$} & \multirow{2}{*}{\textbf{pract.}} \\
& & & & $\min$ & $\max$ & \# & & \\
\midrule
\multirow{1}{*}{10} & 7 & 1 & {$576 \times 440$} & $401$ & 401 & 1 & 1 & 0 \\
\hline
\multirow{1}{*}{9} & 6 & 1 & {$576 \times 448$} & $403$ & 408 & 2 & 2 & 0 \\
\hline
\multirow{2}{*}{8} & 2 & 2 & \multirow{2}{*}{$576 \times 456$} & $409$ & 409 & 1 & 2 & 0 \\
& 5 & 1 & & $408$ & $ 414$ & 5 & 4 & 1 \\
\hline
\multirow{2}{*}{7} & 1 & 2 & \multirow{2}{*}{$576 \times 464$} & $412$ & $415$ & 3 & 5 & 0 \\
& 4 & 1 & & $410$ & $419$ & 6 & 5 & 1 \\
\hline
\multirow{2}{*}{6} & 0 & 2 & \multirow{2}{*}{$576 \times 472$} & $420$ & 420 & 1 & 1 & 0 \\
& 3 & 1 & & $421$ & $ 424$ & 4 & 10 & 0 \\
\hline
\multirow{2}{*}{5} & 2 & 1 & \multirow{2}{*}{$576 \times 480$} & $425$ & 432 & 6 & 7 & 1 \\
& 5 & 0 & & $424$ & $434$ & 3 & 4 & 0 \\
\hline
\multirow{2}{*}{4} & 1 & 1 & \multirow{2}{*}{$576 \times 488$} & $433$ & $438$ & 3 & 6 & 0 \\
& 4 & 0 & & $427$ & $438$ & 3 & 4 & 1 \\
\hline
\multirow{2}{*}{3} & 0 & 1 & \multirow{2}{*}{$576 \times 496$} & $441$ & $443$ & 2 & 4 & 0 \\
& 3 & 0 & & $438$ & $447$ & 7 & 17 & 0 \\
\hline
\multirow{1}{*}{2} & 2 & 0 & $576 \times 504$ & $441$ & $447$ & 6 & 15 & 0 \\
\hline
\multirow{1}{*}{1} & 1 & 0 & $576 \times 512$ & $442$ & $448$ & 7 & 18 & 0 \\
\hline
\multirow{1}{*}{0} & 0 & 0 & $576 \times 520$ & $444$ & $448$ & 4 & 21 & 0 \\
\bottomrule
\end{tabular}
\end{table}

\subsection*{Experiment 6: $\mathbf{T_{245}, r:=\tilde{r}=33}$}
Lastly, \Cref{tab:results_gen_CS_245} shows the results for $T_{245}$ and rank 33. Again the problem has become more difficult because $p$, $n$, and $r$ have increased. 
Still, more numerical solutions are found for several CS parameters compared to the generic case ($(s,t) = (0,0)$). 

\begin{table}[H]
\centering
\caption[Number of exact (numerical and practical) PDs of rank 33 of $T_{245}$, and the size and different ranks of the Jacobian matrix for different combinations of the CS parameters $s$ and $t$ using the AL method from \cite{vermeylen2023stability} on 50 small random starting points.]{Number of exact (numerical and practical) $\mathrm{PD}_{33}\left(T_{245} \right)$s that are obtained for different combinations of $s$ and $t$ using the AL method from \cite{vermeylen2023stability} on 50 random starting points generated as in \eqref{eq:x0}. Also the size and the different ranks of the Jacobian matrix are given.}
\label{tab:results_gen_CS_245}
\begin{tabular}{c | c c | c c c c | c c } 
\multirow{2}{*}{\boldmath $r_{\mathrm{CS}}$} & \multirow{2}{*}{\boldmath ${s}$} & \multirow{2}{*}{\boldmath ${t}$} & \multirow{2}{*}{\boldmath $\mathrm{size}(J_{\mathrm{CS,gen}})$} & \multicolumn{3}{c|}{\boldmath $\mathrm{rank}\left( J_{\mathrm{CS,gen}}(x^*) \right)$} & \multirow{2}{*}{\textbf{\#} \boldmath$x^*$} & \multirow{2}{*}{\textbf{pract.}} \\
& & & & $\min$ & $\max$ & \# & & \\
\midrule
\multirow{1}{*}{10} & 4 & 2 &\multirow{1}{*}{$1600 \times 1174$} & $1091$ & $1091$ & 1 & 1 & 0  \\
\hline
\multirow{3}{*}{9} & 3 & 2 & \multirow{3}{*}{$1600 \times 1182$} & $1088$ & $1096$ & 2 & 2 & 0  \\
& 6 & 1 & & $1089$ & 1099 & 3 & 5 & 0 \\
& 9 & 0 & & $1096$ & $1096$ & 1 & 1 & 0 \\
\hline
\multirow{2}{*}{8} & 5 &1 & \multirow{2}{*}{$1600 \times 1190$} & 1096 & 1105 & 3 & 3 & 0 \\
& 8 & 0 & & $1106$ & $1106$ & 1 & 1 & 0  \\
\hline
\multirow{3}{*}{7} & 1 & 2 & \multirow{3}{*}{$1600 \times 1198$} & $1094$ & $1098$ & 2 & 2 & 0  \\
& 4 & 1 & & $1102$ & $1104$ & 2 & 3 & 0 \\
& 7 & 0 & & $1105$ & 1105 & 1 & 1 & 0 \\
\hline
\multirow{2}{*}{6} & 3 & 1 & \multirow{2}{*}{$1600 \times 1206$} & $1102$ & 1116 & 5 & 5 & 0 \\
& 6 & 0 & & $1115$ & $1115$ & 1 & 2 & 0 \\
\hline
\multirow{2}{*}{5} & 2 & 1 & \multirow{2}{*}{$1600 \times 1214$} & 1118 & 1122 & 3 & 3 & 0 \\
& 5 & 0 & & $1118$ & $1123$ & 4 & 4 & 0 \\
\hline
\multirow{2}{*}{4} & 1 & 1 & \multirow{2}{*}{$1600 \times 1222$} & $1115$ & $1120$ & 4 & 4 & 0 \\
& 4 & 0 & & $1118$ & $1121$ & 3 & 3 & 0 \\
\hline
\multirow{2}{*}{3} & 0 & 1 & \multirow{2}{*}{$1600 \times 1230$} & $1123$ & $1127$ & 3 & 5 & 0  \\
& 3 & 0 & & 1119 & 1131 & 5 & 7 & 0\\
\hline
2 & 2 & 0 & $1600 \times 1238$ & $1123$ & $1128$ & 4 & 7 & 0 \\
\hline
1 & 1 & 0 & $1600 \times 1246$ & $1125$ & $1134$ & 3 & 6 & 0 \\
\hline
0 & 0 & 0 & $1600 \times 1254$ & $1127$ & $1130$ & 3 & 4 & 0 \\
\bottomrule
\end{tabular} 
\end{table}

\setcounter{MaxMatrixCols}{30}

\subsection*{Remark: $\mathbf{T_{333}, r:=\tilde{r}=23}$}
Note that the CS rank can also be chosen smaller than $r$ for $m=p=n$. For example, 11 symmetric rank-1 tensors can be enforced in a $\mathrm{PD}_{23}\left(T_{333} \right)$. The other 12 rank-1 tensors do not have to admit a structure in this case. One of the practical PDs that are found for these parameters has as symmetric part:
\begin{equation*}
\begin{split}
A &= \begin{bNiceMatrix}[r,columns-width=0.4cm]
 1 & 0 & 0 & 0 & 0 & 0 & 0 & 0 & 0 & 0 & 0\\
 0 & 0 & 0 & 1 & 0 & 0 & 0 & 0 &-1 & 0 & 0\\
 0 &-1 &-1 & 0 & 1 & 0 & 0 & 0 & 0 & 0 & 1\\
 0 & 0 &-1 & 0 & 0 & 1 &-1 &-1 & 1 & 0 & 1\\
 0 & 0 & 0 & 1 & 0 & 0 & 1 & 0 &-1 & 0 & 0\\
 0 &-1 &-1 & 0 & 1 & 1 & 0 &-1 & 0 & 0 & 1\\
 0 & 0 & 0 & 0 & 0 & 0 & 0 & 1 & 0 & 0 &-1\\
 0 & 0 & 1 & 0 &-1 &-1 & 0 & 0 & 0 & 1 & 0\\
 0 & 1 & 1 & 0 &-1 &-1 & 0 & 1 & 0 & 1 &-1\\
\end{bNiceMatrix},
\end{split}
\end{equation*}
and as asymmetric part:
\begin{align*}
\dot{U} &= \begin{bNiceMatrix}[r,columns-width=0.4cm]
-1 & 0 & 0 & 0 & 0 &-1 & 0 & 0 & 0 & 0 & 0 & 0\\
 1 & 0 & 0 & 1 & 0 & 0 &-1 & 0 &-1 & 0 & 0 & 0\\
 0 & 0 & 0 & 1 & 0 &-1 & 0 & 0 &-1 &-1 & 0 & 0\\
-1 & 0 & 0 & 0 & 0 &-1 & 0 & 0 & 0 & 0 &-1 & 1\\
 1 & 0 & 0 & 1 & 0 & 0 & 0 & 0 & 0 & 0 & 0 & 0\\
 0 & 0 & 0 & 1 & 0 &-1 & 0 & 1 & 0 & 0 &-1 & 0\\
 1 & 0 & 1 & 0 & 1 & 1 & 0 & 0 & 0 & 0 & 0 & 0\\
-1 & 1 &-1 &-1 & 0 & 0 & 0 & 0 & 0 & 0 & 0 & 0\\
 0 & 0 & 0 &-1 & 0 & 1 & 0 & 0 & 0 & 0 & 0 & 0\\
\end{bNiceMatrix}, \\
\dot{V} &= \begin{bNiceMatrix}[r,columns-width=0.4cm]
 0 & 0 & 0 & 0 & 0 & 0 &-1 & 0 & 0 & 1 & 0 & 0\\
 0 & 0 & 0 & 0 & 0 & 0 & 1 & 1 & 0 & 0 &-1 & 1\\
 0 & 0 & 0 & 0 & 1 & 0 & 0 & 1 & 0 & 1 &-1 & 0\\
 1 & 0 &-1 & 0 & 0 & 0 &-1 & 0 & 0 & 1 & 0 & 0\\
 0 & 0 & 0 & 0 & 0 & 0 & 1 & 1 & 0 & 0 & 0 & 0\\
 0 & 1 &-1 & 0 & 0 & 0 & 0 & 1 & 0 & 1 & 0 & 0\\
 0 & 0 & 0 & 0 & 0 &-1 & 1 & 0 & 1 &-1 & 0 & 0\\
 0 & 0 & 0 &-1 & 0 & 0 &-1 &-1 &-1 & 0 & 0 & 0\\
 0 & 0 & 0 & 0 & 0 & 0 & 0 &-1 & 0 &-1 & 0 & 0\\
\end{bNiceMatrix},  
\end{align*}
\begin{align*}
\dot{W} &= \begin{bNiceMatrix}[r,columns-width=0.4cm]
 0 & 0 & 0 & 0 & 1 & 0 & 0 & 0 & 0 & 0 & 0 & 1\\
-1 & 1 &-1 & 0 & 0 & 0 & 0 & 0 & 0 & 0 & 0 &-1\\
 0 & 1 &-1 & 0 & 1 & 1 & 0 & 0 & 0 & 0 & 0 & 0\\
 0 & 0 & 0 & 0 & 1 & 0 & 1 & 0 &-1 & 0 & 0 & 1\\
 0 & 1 & 0 & 0 & 0 & 0 & 0 & 0 & 0 & 0 & 0 &-1\\
 0 & 1 & 0 &-1 & 1 & 0 & 0 & 0 &-1 & 0 & 0 & 0\\
 0 & 0 & 0 & 0 &-1 & 0 & 0 & 0 & 0 &-1 & 1 &-1\\
 0 &-1 & 0 & 0 & 0 & 0 & 0 & 1 & 0 & 0 &-1 & 1\\
 0 &-1 & 0 & 0 &-1 & 0 & 0 & 0 & 0 & 0 & 0 & 0\\
\end{bNiceMatrix}. 
\end{align*}


\section{Conclusion}
The proposed generalized cyclic symmetric structure can be useful to include in the optimization problem for the \emph{fast matrix multiplication (FMM)} problem to reduce the number of parameters and to improve the convergence. More numerically exact and practical \emph{polyadic decompositions (PDs)} of \emph{matrix multiplication tensors (MMTs)} can be obtained with this structure. Several new \emph{practical} PDs, i.e., sparse PDs with elements that are either one, minus one, or zero, with this new structure are discovered for various problem parameters, i.e., sizes of the MMT. 

\section{Acknowledgement}

This work was funded by (1) the Flemish Government: this research received funding under the AI Research Program. Charlotte Vermeylen is affiliated to Leuven.AI - KU Leuven institute for AI, B-3000, Leuven, Belgium. (2) KU Leuven Internal Funds: iBOF/23/064, C14/22/096 and
IDN/19/014.
Marc Van Barel was supported by the Research Council KU Leuven, C1-project C14/17/073 and by the Fund for Scientific Research–Flanders (Belgium), EOS Project no 30468160 and project G0B0123N.

\section{Competing interests}

Declarations of interest: none.


\bibliographystyle{elsarticle-num} 
\bibliography{allpapers}


%
%
%

\end{document}